\documentclass[11pt,a4paper]{article}
\usepackage[utf8]{inputenc}
\usepackage[english]{babel}
\usepackage[T1]{fontenc}
\usepackage{lmodern}
\usepackage{amsmath}
\usepackage{accents}
\usepackage{amsfonts}
\usepackage{amssymb}
\usepackage{amsthm}
\usepackage{amscd}
\usepackage{url}
\usepackage{enumitem}
\usepackage{hhline}
\usepackage{dsfont}
\usepackage{array}
\usepackage{geometry}
\usepackage{tikz}
\usetikzlibrary{decorations.pathreplacing}
\usetikzlibrary{patterns}
\usepackage{xcolor}
\usetikzlibrary{arrows}
\usepackage{mdframed}
\usepackage{caption}
\usepackage[titletoc]{appendix}
\usepackage{csquotes}
\usepackage{hyperref}
\usepackage{float}
\usepackage{bm}
\usepackage{microtype}
\usepackage{graphicx}

\mdfsetup{nobreak=true}

\AtBeginDocument{
  \label{CorterrectFirstPageLabel}
  
}

\usepackage[backend=biber, doi=false,isbn=false,url=false]{biblatex}
\newcolumntype{L}{>{$}l<{$}}

\newcommand\N{\mathbb{N}}
\newcommand\Nn{\mathcal{N}}
\newcommand\M{\mathcal{M}}

\newcommand\J{\mathcal{J}}

\newcommand\R{\mathbb{R}}
\newcommand\NN{\mathsf{N}}
\newcommand\oNN{\overline{\mathsf{N}}}

\newcommand\x{\underline{x}}

\newcommand\z{\underline{z}}
\newcommand\w{\underline{w}}

\newcommand\n{\underline{n}}
\newcommand\m{\underline{m}}
\newcommand\p{\underline{p}}

\newcommand\y{\underline{y}}

\newcommand\uphi{\underline{\phi}}

\newcommand\Ss{\mathcal{S}}
\newcommand\E{\mathbb{E}}
\newcommand\Id{\mathrm{Id}}
\DeclareMathOperator\Var{Var}
\DeclareMathOperator\supp{supp}

\newcommand\dd{\mathrm{d}}
\newcommand\CC{\mathcal{C}}

\newcommand\B{\mathcal{B}}
\newcommand\G{\mathcal{G}}
\newcommand\I{\mathcal{I}}

\newcommand\PP{\mathbb{P}}
\newcommand\Pt{\mathcal{P}}

\DeclareMathOperator\Law{\mathrm{Law}}

\DeclareMathOperator\Cov{Cov}
\newcommand\Card{\mathrm{Card}}

\newcommand\T{\mathbb{T}}
\DeclareMathOperator\Tr{Tr}

\newcommand{\enstq}[2]{\left\{#1~\middle|~#2\right\}}

\DeclareMathOperator\sinc{sinc}

\newcommand\one{\mathds{1}}

\newcommand\numberthis{\addtocounter{equation}{1}\tag{\theequation}}
\newcommand\jump{\par\medskip}
\newcommand\quand{\quad\text{and}\quad}

\theoremstyle{plain}
\newtheorem{theo}{Theorem}[section]
\newenvironment{theorem}%
  {\begin{mdframed}[backgroundcolor=white]\begin{theo}}%
  {\end{theo}\par\vspace{0.1cm}\end{mdframed}}
  
\theoremstyle{plain}
\newtheorem{coro}[theo]{Corollary}
\newenvironment{corollary}%
  {\begin{mdframed}[backgroundcolor=white]\begin{coro}}%
  {\end{coro}\par\smallskip\end{mdframed}}
  
\theoremstyle{plain}
\newtheorem{lemm}[theo]{Lemma}
\newenvironment{lemma}%
  {\begin{mdframed}[backgroundcolor=white]\begin{lemm}}%
  {\end{lemm}\par\vspace{0.cm}\end{mdframed}}

\theoremstyle{plain}
\newtheorem{prop}[theo]{Proposition}
\newenvironment{proposition}%
  {\begin{mdframed}[backgroundcolor=white]\begin{prop}}%
  {\end{prop}\par\vspace{0.cm}\end{mdframed}}

\theoremstyle{definition}
\newtheorem{remark}[theo]{Remark}

\makeatletter
\def\blfootnote{\gdef\@thefnmark{}\@footnotetext}
\makeatother
  
\makeatletter
\renewcommand*\env@matrix[1][*\c@MaxMatrixCols c]{%
  \hskip -\arraycolsep
  \let\@ifnextchar\new@ifnextchar
  \array{#1}}
\makeatother

\begin{document}

\title{\Huge{Cumulants asymptotics for the zeros counting measure of real Gaussian processes}}
\author{Louis Gass\\\strut\\Email : louis.gass@ens-rennes.fr}
\maketitle
\blfootnote{Univ Rennes, CNRS, IRMAR - UMR 6625, F-35000 Rennes, France.}
\blfootnote{This work was supported by the ANR grant UNIRANDOM, ANR-17-CE40-0008.}
\blfootnote{Email: louis.gass(at)ens-rennes.fr}

\begin{abstract} We compute the exact asymptotics for the cumulants of linear statistics associated with the zeros counting measure of a large class of real Gaussian processes. Precisely, we show that if the underlying covariance function is regular and square integrable, the cumulants of order higher than two of these statistics asymptotically vanish. This result implies in particular that the number of zeros of such processes satisfies a central limit theorem. Our methods refines the recent approach by M. Ancona and T. Letendre and allows us to prove a stronger quantitative asymptotics, under weaker hypotheses on the underlying process. The proof exploits in particular the elegant interplay between the combinatorial structures of cumulants and factorial moments in order to simplify the determination of the asymptotics of nodal observables. The class of processes addressed by our main theorem includes as motivating examples random Gaussian trigonometric polynomials, random orthogonal polynomials and the universal Gaussian process with sinc kernel on the real line, for which the asymptotics of higher moments of the number of zeros were so far only conjectured. 
\end{abstract}

\renewcommand\contentsname{} 

\begingroup
\let\clearpage\relax
\vspace{-1cm} 
\setcounter{tocdepth}{2}
\tableofcontents
\endgroup


\newpage
\section{Introduction}
The study of the number of zeros of smooth Gaussian processes has a long history and is in particular motivated by the pioneering works of Kac and Rice, see e.g. \cite{Aza09} for a general introduction to this topic. The asymptotics for the expectation and the variance of the number of zeros of a stationary Gaussian process on an interval growing interval $[0,R]$ as $R$ grows to infinity has been known since \cite{Cuz76}, where a central limit theorem (CLT) for the number of zeros is also proved. The variance asymptotics is there established  using the celebrated Kac--Rice method and the CLT is proved using approximation by an $m$-dependent process.\jump

With similar methods, the variance of the number of zeros of random Gaussian trigonometric polynomials with large degree has been studied in \cite{Gra11}, as well as the associated CLT. Later on, the machinery of Wiener chaos expansion was then successfully used in order to compute the variance asymptotics as well as establishing CLTs for the number of zeros of various models of stochastic processes, see for instance \cite{Aza13,Arm21,Do21}. Central limit theorems for the number of real roots of random algebraic polynomials have also been investigated, see for example \cite{Ngu21} and the references therein.\jump

In the recent paper \cite{Gas21}, focusing on the asymptotics of the Kac density rather than on the full integral Kac--Rice formula, the author managed to avoid some of the technical computations inherent to the use of Kac--Rice method. This allowed him to get a unifying point of view, make explicit the needed decorrelation  estimates and then deduce the variance asymptotics  for the number of zeros of many models of Gaussian processes. It has been then conjectured that the same heuristics could be applied to treat the asymptotics of the higher central moments of the number of zeros of a Gaussian process, which is the goal of the present paper.\jump

Up to now, very few results about the asymptotics of higher central moments are known. The best result so far is the one by M. Ancona and T. Letendre  \cite{Anc21}, where it has been proved that the $p$-th central moment, when properly rescaled, converges towards the $p$-th moment of a Gaussian random variable, under restrictive condition that the covariance function and their derivatives decreases faster than $x^{-4p}$. This last result then yields another proof of the CLT for the number of zeros by the method of moments, for the processes whose covariance function is in the Schwartz class of regular and rapidly decreasing functions. Their proof is based on a series of articles \cite{Anc19, Anc20} whose purpose was to tackle the CLT for the number of roots of Kostlan polynomials and its real algebraic extension. Note that the \textit{sinc process}, i.e. the Gaussian process with $\sinc$ covariance function, which plays a central role in probability theory and mathematical physics, is ruled out from their framework, due to the slow decay of the $\sinc$ kernel. In the more general context of point processes, higher moments of geometric statistics have also been studied under the hypothesis of fast decreasing correlation \cite{Bla19}.\jump

In this paper, we prove the exact asymptotics of the higher central moments of the number of zeros of a large class of Gaussian processes, under the only hypothesis, apart from regularity, that the covariance function as well as its derivatives are square integrable. Our results apply in particular to Gaussian trigonometric and orthogonal polynomials, as well as the stationary process with $\sinc$ kernel and other Gaussian stationary processes on the real line with possibly slow decaying kernels. We prove in fact a more general theorem by computing the exact asymptotics of the cumulants of linear statistics associated with the zeros counting measure of the underlying processes. The use of cumulants instead of central moments simplifies the rather intricate combinatorics involved when estimating higher order moments via the Kac--Rice method.\jump

Our result in turn implies the convergence of associated moments of any order and thus a CLT, with an exact rate of convergence. As a corollary, we deduce a polynomial concentration of any order of the number of zeros, and by a Borel--Cantelli argument, the almost sure convergence of the number of zeros. Note that these last facts cannot be deduced from chaos expansion methods. More generally, in the context of linear statistics, we prove the almost sure equidistribution of the zeros set at the limit for a large class of smooth Gaussian processes.

\subsection{Statement of the main results}
\subsubsection{Cumulants asymptotics and central limit theorems}
In the following, all the random variables considered are defined on a common abstract probability space $(\Omega, \mathcal F, \mathbb P)$ and $\mathbb E$ will denote the associated expectation. In the sequel, $W$ stands for a standard Gaussian random variable, i.e. centered with unit variance. We denote by $\kappa_p(Z)$ the $p$-th cumulant of a random variable $Z$, given by the expression
\[\kappa_p(Z) = \sum_{\I\in\Pt_p} (-1)^{|\I|-1}(|\I|-1)!\prod_{I\in\I}\E\left[Z^{|I|}\right],\numberthis\label{eqter:69}\]
where the sum is indexed by the set $\Pt_p$ of all the partitions of the finite set $\{1,\ldots,p\}$. We refer to \cite{Spe83, Pec11} and the paragraph \ref{subsecter4} below for more details on the cumulants of a random variable. \jump

In the following, iid stands for independent and identically distributed. The following theorem describes the asymptotics of all the cumulants of the number of zeros of a Gaussian trigonometric polynomial with independent coefficients.
\begin{theorem}
\label{Theoremter1}
Let $(a_k)_{k\geq 0}$ and $(b_k)_{k\geq 0}$ be two iid sequences of standard Gaussian variables. Let $Z_n$ be the number of zeros on $[0,2\pi]$ of the process
\[h_n(x) := \frac{1}{\sqrt{n}}\sum_{k=0}^{n-1}a_k\cos(kx) + b_k\sin(kx).\]
For $p$ a positive integer, there is an explicit finite constant $\gamma_p$ such that
\[\lim_{n\rightarrow+\infty}\frac{\kappa_p(Z_n)}{n} = \gamma_p.\]
\end{theorem}
\smallskip
The constants $\gamma_1$ and $\gamma_2$ are positive (i.e. $>0$). The above theorem implies in particular that
\[\lim_{n\rightarrow+\infty}\frac{\Var(Z_n)}{n}= \gamma_2 \quand\forall p\geq 3,\quad \frac{\kappa_p(Z_n)}{n^{p/2}} =O\left(\frac{1}{n^{\frac{p}{2}-1}}\right).\numberthis\label{eqter:70}\]
Given the expression of the central moments in terms of cumulants and the fact that the cumulants of a Gaussian random variable are zeros for $p\geq 3$, the asymptotics \eqref{eqter:70} imply in fact that for every positive integer $p$,
\[\E\left[ \left(\frac{Z_n-\E[Z_n]}{\sqrt{\Var(Z_n)}}\right)^p\right] = \E[W^p] + O\left(\frac{1}{\sqrt{n}}\right).\numberthis\label{eqter:55}\]
Note that the exact asymptotics of the cumulants given by Theorem \ref{Theoremter1} is in nature stronger than the cruder bound given by \eqref{eqter:70} and thus the central moment asymptotics \eqref{eqter:55}. \jump

As a consequence, we are able to reprove the central limit theorem for the number of zeros, as well as a polynomial concentration to any order of the number of zeros around its mean by Markov inequality.
\begin{corollary}
\label{Corter4}
As $n$ goes to infinity, we have the convergence in distribution
\[\frac{Z_n-\E[Z_n]}{\sqrt{\Var(Z_n)}} \underset{n\rightarrow+\infty}{\overset{d}{\longrightarrow}} \mathcal{N}(0,1).\]
For all $p\geq 2$, there is a constant $C_p$ such that for all integer $n$ and positive constant $\omega$,
\[\PP\left(\left|Z_n-\E[Z_n]\right|\geq n\omega\right)\leq \frac{C_p}{(\sqrt{n}\omega)^p}.\]
\end{corollary}

Note that the variance estimate in Equation \eqref{eqter:70} and the associated CLT were first established in \cite{Gra11} by the Kac--Rice method and in \cite{Aza13} by the Wiener chaos expansion. So far the exact asymptotics of the $p$-th central moment or cumulants of $Z_n$ has never been computed for $p\geq 3$. Theorem \ref{Theoremter1} shows that it asymptotically behaves like the $p$-th moment of a Gaussian random variable, which is expected from the already existing central limit theorem for the random variable $Z_n$. The polynomial concentration of the number of zeros and a Borel--Cantelli argument implies the almost sure convergence
\[\lim_{n\rightarrow +\infty} \frac{Z_n}{n} = \gamma_1\quad\text{a.s},\]
a result already known from \cite{Ang21}, using a derandomization method. Exponential concentration has been established in \cite{Ngu19} for this particular model but the proof is of very different nature and strongly use the trigonometric nature of the random process $h_n$. Our proof only uses the fact that the process is of class $\CC^\infty$ and is adaptable to many other models.\jump 

The error term in \eqref{eqter:55} is new and implies a rate of convergence towards the Gaussian random variable of order $1/\sqrt{n}$ for the moment metric. It is reminiscent of the Berry--Essen bound for the more classical CLT. Note that the Wiener chaos expansion method can yield (slower) speeds of convergence for more classical distances, namely Kolmogorov or Wasserstein. \jump

The independence hypothesis above on the Gaussian random coefficients can be relaxed. Namely, we can extend the previous Theorem \ref{Theoremter1} to the case where the Gaussian sequences $(a_k)_{k\geq 0}$ and $(b_k)_{k\geq 0}$ are independent and stationary.
\begin{theorem}
\label{Theoremter2}
Let $(a_k)_{k\geq 0}$ and $(b_k)_{k\geq 0}$ be two independent sequences of standard Gaussian variables, such that
\[\E[a_ka_l] = \E[b_kb_l] = \rho(k-l).\]
We assume that the spectral measure associated with the correlation function $\rho$ has a continuous positive density on the torus $\T$. Let $Z_n$ be the number of zeros on $[0,2\pi]$ of the process
\[h_n(x) := \frac{1}{\sqrt{n}}\sum_{k=0}^{n-1}a_k\cos(kx) + b_k\sin(kx).\]
Then the conclusion of Theorem \ref{Theoremter1} and its Corollary \ref{Corter4} holds.
\end{theorem}

The expectation of the number of zeros in this model has been studied in \cite{Ang19,Ang21b} and the variance in \cite{Gas21}. The above Theorem \ref{Theoremter2} gives the asymptotics of every cumulant and therefore, as discussed above in the independent case, it proves a central limit theorem for the number of zeros, which is a new result in this dependent framework, as well as concentration around the mean and a quantification of the rate of convergence.\jump

In another direction, one can replace the functions $\cos$ and $\sin$ by more general functions. A standard framework is then the following model of random orthogonal polynomials, for which we can give a similar statement.
\begin{theorem}
\label{Theoremter3}
Let $(a_k)_{k\geq 0}$ be an iid sequence of standard Gaussian variables. Let $(P_k)_{k\geq 0}$ be a sequence of orthogonal polynomials associated with a measure $\mu$ on the line, and let $[a',b']$ be an interval. We assume that the measure $\mu$ and the interval $[a',b']$ satisfies the hypotheses of \cite[Thm.~1.1]{Do21}. Let $Z_n$ be the number of zeros on $[a',b']$ of the process
\[h_n(x) := \frac{1}{\sqrt{n}}\sum_{k=0}^{n-1}a_kP_k(x).\]
Then the conclusion of Theorem \ref{Theoremter1} and its Corollary \ref{Corter4} holds.
\end{theorem}

The expectation, the variance and a central limit theorem for this model have been very recently studied with the Wiener chaos expansion method in \cite{Do21}. Here again, we extend this result by determining the asymptotics of higher cumulants and thus higher moments. As already discussed after  the previous statements above, from the cumulants asymptotics established in Theorem \ref{Theoremter3}, we can also deduce concentration around the expected number of zeros and as well as a rate of convergence in the associated CLT for the (non-standard) metric of moments.\jump

At last, we extend known results about the number of zeros of a stationary Gaussian process on a growing interval, establishing in particular a CLT under the sole square integrability of the associated correlation function and its derivatives.
\begin{theorem}
\label{Theoremter8}
Let $f$ be a stationary Gaussian process with $\CC^\infty$ paths and covariance function $r$.
For $R>0$ we define $Z_R$ to be the number of zeros on $[0,R]$ of the process $f$. \begin{itemize}
\item If the covariance function $r$ and its derivatives are in $L^2(\R)$, then for all $p\geq 2$
\[\lim_{R\rightarrow +\infty} \E\left[ \left(\frac{Z_R-\E[Z_R]}{\sqrt{\Var(Z_R)}}\right)^p\right] = \E[W^p].\]
\item If the covariance function $r$ and its derivatives are in $L^q(\R)$ for all $q> 1$, then for $p$ a positive integer, there is an explicit finite constant $\gamma_p$ such that
\[\lim_{R\rightarrow+\infty}\frac{\kappa_p(Z_R)}{R} = \gamma_p.\]
\end{itemize}
\end{theorem}

As mentioned above, the CLT which is obtained from the above moments asymptotics by the method of moments is already known in the particular case of stationary Gaussian processes with covariance function belonging to the Schwartz class, see \cite{Anc21}. Here the assumption on the decay of the correlation function is greatly relaxed and we only need to assume the square integrability of the covariance kernel as well as its derivatives. It implies a polynomial concentration around the mean to any order for the number of zeros, which appears to be a new result for regular process with slow decaying covariance functions (note that exponential concentration has been proved in \cite{Bas20} under some analyticity assumption).\jump

As a particular and representative case, Theorem \ref{Theoremter8} covers the example of the stationary Gaussian process $f$ with $\sinc$ kernel, which is a completely new result. This process plays a central role in the study of determinantal point processes, and appears as the limit of the local statistics of various random models, from eigenvalues of random matrices to random integer partitions. For this particular process, the asymptotic of the expectation and the variance of $Z_R$, as well as a CLT were known since the pioneering works of \cite{Cuz76} and the references therein. But so far, the exact asymptotics of higher central moments or cumulants of $Z_R$ remained unknown.\jump

Under the stronger hypothesis that the covariance function and its derivatives are in $L^p$ for all $p>1$, we deduce the exact asymptotic of the cumulants of any order. This integrability hypothesis in particular holds true for processes whose covariance functions $r$ and their derivatives satisfy the bound
\[\forall x\in\R,\quad r^{(u)}(x)\leq \frac{C_u}{1+|x|},\]
which is the case for a stationary Gaussian process with $\sinc$ covariance function.
\subsubsection{A more general and unifying  statement}
In fact Theorems \ref{Theoremter1}, \ref{Theoremter2}, \ref{Theoremter3} and \ref{Theoremter8} are all corollaries of a single, more general statement given below. In order to state it, we need to introduce first a few notations that will be used for the rest of the paper.\jump

Let $U$ be a non-empty open interval of the real line $\R$ or of the one-dimensional torus $\T$, endowed with their canonical distance $|\,.\,|$. Let $n\in\R_+^*\cup\{+\infty\}$. If $n$ is finite then $nU$ is a non-empty open subset of $\R$ or of the one-dimensional torus $n\T$ of length $n$. For $n=+\infty$ we use the convention $(+\infty)U = \R$. This setting allows us to give a unified exposition for processes defined on the torus (e.g. random trigonometric polynomials) and on the real line (e.g. the sinc process).\jump

Let $\NN$ be an unbounded subset of $\R_+^*$ and $\oNN = \NN\sqcup\{+\infty\}$.  For each $n\in\oNN$, we consider a centered Gaussian process $f_n$ defined on $nU$, and we assume that the process $f_\infty$ is a non-zero stationary centered process on $\R$. Note that for $n\in\NN$ the process $h_n = f_n(n\,.\,)$ is a Gaussian process on $U$. For $n\in \oNN$ and $s,t\in nU$ we define the covariance function
\[r_n(s,t) = \E[f_n(s)f_n(t)]\quand r_\infty(s-t) = r_\infty(s,t).\]
If the process $f_n$ is of class $\CC^k(U)$ for $k\geq 0$ then the covariance function $r_n$ is also of class $\CC^k$ in each variable, and one has for $u,v\leq k$ and $x,y\in nU$
\[r_n^{(u,v)}(x,y) = \E[f_n^{(u)}(x)f_n^{(v)}(y)].\]
For $n\in\NN$ we define
\[Z_n = \enstq{x\in U}{f_n(nx) = 0}\quand \nu_n:= \sum_{x\in Z_n}\delta_x,\numberthis\label{eqter:52}\]
the random counting measure on $Z_n$. Note that $(\nu_n)_{n\in\NN}$ is a family of measures on $U$. Assume for now (it will be a consequence of Bulinskaya Lemma) that for each $n\in\NN$ the set $Z_n$ is almost surely locally finite. For a bounded function $\phi:U\rightarrow\R$, with compact support in $U$, we define the bracket
\[\langle \nu_n, \phi\rangle = \sum_{x\in Z_n} \phi(x).\]
For instance, if $U=\R$ and $\phi = \one_{[0,1]}$ then 
\[\langle \nu_n, \phi\rangle = \Card \enstq{x\in [0,n]}{f_n(x) = 0}.\]
Note the Kac--Rice formula implies that in expectation, the counting measure has a density with respect to the Lebesgue measure. It means that one can compute, when it is defined, expectations of linear statistics for test functions defined almost everywhere.
For $q\geq 1$ we define the two following hypotheses.
\begin{itemize}
\item[$\bullet$] $H_1(q)\,:$ The sequence of processes $(f_n)_{n\in\oNN}$ is of class $\CC^q(U)$, and there is a uniformly continuous function $\psi$ on $U$, bounded below and above by positive constants, such that for $u,v\leq q$, the following convergence holds uniformly for $x\in U$ and locally uniformly for $s,t\in \R$,
\[\lim_{n\rightarrow +\infty} r_n^{(u,v)}(nx+s,nx+t) = \psi(x) r_\infty^{(u,v)}(s,t).\numberthis\label{eqter:64}\]
\item[$\bullet$] $H_2(q)\,:$ There is a function $g$, even, bounded and going to zero near infinity, such that for $u,v\leq q$, $n\in \oNN$ and $s,t\in nU$,
\[|r_n^{(u,v)}(s,t)|\leq g(s-t),\numberthis\label{eqter:65}\]
and for some positive constant $\omega$ the function $g_\omega$ is in $L^2(\R)$, where
\[g_\omega :x\;\mapsto \;\sup_{|u|\leq \omega} g(x+u).\]
\end{itemize}
\begin{theorem}
\label{Theoremter9}
Let $p\geq 2$ and $q=2p-1$. We assume that the sequence of processes $(f_n)_{n\in\oNN}$ satisfies hypotheses $H_1(q)$ and $H_2(q)$ defined above. Then for every function $\phi\in L^1\cap L^{p^2}(U)$,
\[\lim_{n\rightarrow+\infty} \E\left[\left(\frac{\langle \nu_n, \phi\rangle - \E\left[\langle \nu_n, \phi\rangle\right]}{\sqrt{\Var\left(\langle \nu_n, \phi\rangle\right)}}\right)^p\right] = \E[W^p].\]
Assume moreover that $g_\omega\in L^{\frac{p}{p-1}}(\R)$. Then there is an explicit constant $\gamma_p$ depending only on the process $f_\infty$, such that
\[\lim_{n\rightarrow+\infty} \frac{\kappa_p(\langle\nu_n,\phi\rangle)}{n} = \gamma_p\,.\,\left(\int_{U} \phi(x)^p\dd x\right).\]
\end{theorem}

The assumption $H_1(q)$ characterizes the convergence of the family of processes $(f_n)_{n\in\NN}$ towards a limit stationary process in $\CC^q$ norm.  This hypothesis is natural and arises in many models. For instance the covariance function of random trigonometric polynomials converges towards the $\sinc$ function. The regularity of the process $f_n$ ensures the well-definiteness of the $p$-th moment, see for instance \cite[Thm 3.6]{Aza09}. The convergence towards a non-degenerate stationary process ensures the uniform non-degeneracy on the process $f_n$, as well as the explicit asymptotics for the cumulants.
\jump

The decay assumption in $H_2(q)$ is greatly relaxed compared to the one present in \cite{Anc21}, where the authors require a function $g$ that decrease like $x^{-4p}$ (though they need only to take $q=p-1$ in Theorem \ref{Theoremter9}). Here we show that the asymptotics of higher moments is independent of the rate of decay of the covariance function, and must only satisfy some uniform square integrability condition. The number of finite moments (and their asymptotics) that one can obtain is directly related to the regularity of the process.\jump

Let us briefly now show that the unifying Theorem \ref{Theoremter9} indeed implies the collection of theorems of the previous subsection. First, Theorem \ref{Theoremter1} and \ref{Theoremter2} are a consequence of Theorem \ref{Theoremter9}, by setting $U=\T$, $\NN = \N^*$, $\phi = \one_\T$ and
\[f_n(x) = \frac{1}{\sqrt{n}}\sum_{k=0}^{n-1}a_k\cos\left(\frac{kx}{n}\right) + b_k\sin\left(\frac{kx}{n}\right).\]
Let $\psi$ be the spectral density of the correlation function $\rho$ of the stationary Gaussian sequences $(a_k)_{k\geq 0}$ and $(b_k)_{k\geq 0}$, which is assumed to be continuous and positive on $\T$. Assumptions $H_1(q)$ and $H_2(q)$ are proved for all $q>0$ for this model in the paper \cite{Gas21} with limit process having $\sinc$ covariance function, and 
\[g = \frac{C_\alpha}{1+|x|^\alpha},\]
where the exponent $\alpha$ can be taken in $]1/2,1[$. Note that Theorem \ref{Theoremter1} is a particular case of Theorem \ref{Theoremter2} with $\psi = 1$ (in that case, one can take $\alpha=1$ above).  \jump

Similarly, Theorem \ref{Theoremter3} is a consequence of Theorem \ref{Theoremter9}. Let $\mu$ be a measure with compact support on the real line. We set $U$ a subinterval of $\R$ such that $\mu$ has a positive continuous density on $\overline{U}$. It is proved in \cite{Do21} under mild assumption on the measure $\mu$ that for the model of random orthogonal polynomials with respect to the measure $\mu$, the assumption $H_2(q)$ holds true for all $q>0$ and 
\[g(x) = \frac{C}{1+|x|}.\] 
Let $\omega$ the density of the equilibrium measure of the support of $\mu$, which is continuous and positive on $U$ and $\psi$ the inverse of the density of the measure $\mu$. Then a variation of hypothesis $H_1(q)$ holds true for all $q>0$ with limit process having $\sinc$ covariance function, where \eqref{eqter:64} statement should be instead
\[\lim_{n\rightarrow +\infty} \left(\frac{1}{\omega(x)}\right)^{u+v}r_n^{(u,v)}\left(nx+\frac{s}{\omega(x)},nx+\frac{t}{\omega(x)}\right) = \psi(x) r_\infty^{(u,v)}(s,t).\]
The proof of Theorem \ref{Theoremter9} adapts verbatim to this setting and the conclusion is
\[\lim_{n\rightarrow+\infty}\frac{\kappa_p(\langle\nu_n,\phi\rangle)}{n} = \gamma_p\,.\,\left(\int_{U} \phi(x)^p\omega(x)\dd x\right).\]
Note that if $\supp \mu = [0,1]$, then after a change of variable, the equilibrium measure is simply the Lebesgue measure on the torus $\T$ and hypothesis $H_1(q)$ then exactly holds true.\jump

At last, Theorem \ref{Theoremter8} is again a consequence of Theorem \ref{Theoremter9} with $U = \R$, $\NN = \R_+^*$, $f_n = f$, and test function $\phi = \one_{[0,1]}$.\jump
\subsubsection{Asymptotics for the linear statistics}
Let $\nu_\infty$ denote the Lebesgue measure on the interval $U$. Theorem \ref{Theoremter9} implies a strong law of large number and a central limit theorem for the sequence of random measure $(\nu_n)_{n\in\NN}$. The two following Corollaries \ref{Corter6} and \ref{Corter7} extend the results of \cite[Sec. 1.4]{Anc21} to our framework, and we refer to this paper for a more thorough discussion.
\begin{corollary}[Law of large numbers]
\label{Corter6}
Assume that the hypotheses $H_1(q)$ and $H_2(q)$ are satisfied for all $q\geq 1$, and either $\NN = \N^*$, or $\NN = \R_+^*\;$  and for $n\in\R_+^*$, $f_n = f_\infty$. Then we have the following almost-sure convergence for the vague topology
\[\lim_{n\rightarrow +\infty} \frac{1}{n}\nu_n = \gamma_1\,\nu_\infty\quad \mathrm{a.s.}\quad.\]
\end{corollary}

Corollary \ref{Corter6} shows that zeros of the process $f_n(n\,.\,)$ tend to be equidistributed on the set $U$ as $n$ goes to $+\infty$. When $\NN = \N^*$, the proof follows from an application of the Borel--Cantelli Lemma. When $\NN = \R_+^*\;$  and $\;\;\forall n\in\R_+^*,\;\;f_n = f_\infty$, we can apply the Borel--Cantelli Lemma to prove the almost sure convergence on a polynomial subsequence. It is then a standard fact that the monotonicity of $Z_n$ ensures the almost sure convergence of the whole sequence. 

\begin{corollary}[Central limit theorem]
\label{Corter7}
Assume that the hypotheses $H_1(q)$ and $H_2(q)$ are satisfied for all $q\geq 1$. Then we have the following convergence in distribution
\[\forall \phi\in L^1\!\cap \!L^\infty(U),\quad\frac{\sqrt{n}}{\gamma_2}\left\langle \left(\frac{1}{n}\nu_n - \gamma_1\nu_\infty\right), \phi\right\rangle \underset{n\rightarrow+\infty}{\sim} \Nn\left(0,\|\phi\|_2^2\right).\]
\end{corollary}

Corollary \ref{Corter7} implies that the fluctuations around the mean of the counting measure $\nu_n$ is comparable to a Gaussian white noise.
\subsection{Outline of the proof}
Before giving a complete and detailed proof of Theorem \ref{Theoremter9}, let us sketch its main ingredients and arguments. The proof follows a similar strategy as in \cite{Anc21} but with considerable refinements. It mainly relies on a careful analysis of the Kac--Rice formula, which asserts that for a test function $\phi$,
\[\E[\langle\nu_n,\phi\rangle^p] = \int_{(nU)^p} \left(\prod_{i=1}^p\phi\left(\frac{x_i}{n}\right)\right)\rho_{p,n}(x_1,\ldots,x_p)\dd x_1\ldots\dd x_p\;+\;\text{extra terms},\numberthis\label{eqter:02}\]
where 
\[\rho_{p,n}(x_1,\ldots,x_p) := T_n(x_1,\ldots,x_p)\E\left[\prod_{i=1}^p|f_n'(x_i)|\,\middle|\,f_n(x_1) = \ldots = f_n(x_p) = 0\right],\numberthis\label{eqter:56}\]
and $T_n(x_1,\ldots x_p)$ is the density at zero of the Gaussian vector $(f_n(x_1),\ldots,f_n(x_p))$. The extra terms appearing in Equation \eqref{eqter:02} are of combinatorial nature and can be treated the exact same way as the first term, so we will omit them in the following heuristics. The function $\rho_{p,n}$ is called the Kac density of order $p$ associated with the process $f_n$. Observe that the function $\rho_{p,n}$ is ill-defined when two of its arguments collapse. This issue is solved by using the technique of divided differences, that appeared in \cite{Cuz75} and was subsequently developed in \cite{Anc19, Anc20, Anc21}. Let us give an example with $p=2$. The idea is to replace in \eqref{eqter:56} the quantity 
\[f_n(x_1) = f_n(x_2) = 0\quad\quad\text{by}\quad\quad f_n(x_1) = \frac{f_n(x_2)-f_n(x_1)}{x_2-x_1} = 0.\]
If the variables $x_1$ and $x_2$ collapse, the second expression  becomes $f_n(x_1) = f_n'(x_1) = 0$. The regularity of the process $f_n$ given by the assumption $H_1(q)$ and the non-degeneracy of the limit process $f_\infty$ implies that the Gaussian vector $(f_n(x),f_n'(x))$ is non-degenerate and gives an alternative non-singular expression of the function $\rho_{2,n}$ near the diagonal. For higher integers $p$, the reasoning is the same. For each partition $\I$ of the set $\{1,\ldots,p\}$, we will thus give an alternative and non-singular expression of the density $\rho_{p,n}$, that extends by continuity on points $(x_1,\ldots,x_p)$ such that $x_i$ and $x_j$ are equal if $i$ and $j$ belong to the same cell of the partition $\I$. This procedure is explained in Section \ref{subsecter5}.\jump

From now the proof is considerably refined compared to \cite{Anc21}, where we rather use the powerful combinatorics of cumulants to simplify and enhance the results. Developing the expression of the cumulant of order $p$ as a function of the moments, we get
\begin{align*}
\kappa_p(\langle\nu_n,\phi\rangle) &= \sum_{\J}(|\J|-1)!(-1)^{|\J|-1}\prod_{J\in\J}\E\left[\langle \nu_n, \phi\rangle^{|J|}\right]\\
&= \int_{(nU)^p} \prod_{i=1}^p\phi\left(\frac{x_i}{n}\right) F_{p,n}(x_1\ldots,x_p)\dd x_1\ldots\dd x_p + \;\text{extra terms},\numberthis\label{eqter:57}
\end{align*}
where the sum indexed by $\J$ runs over all the partitions of the set $\{1,\ldots,p\}$, and with
\[F_{p,n}(x_1\ldots,x_p) = \sum_{\J}(|\J|-1)!(-1)^{|\J|-1}\prod_{J\in\J}\rho_{|J|,n}(\x_J).\]

Now let $\I$ be a partition of $\{1,\ldots,p\}$, and assume that for $i$ and $j$ belonging to two different cells of the partition $\I$, the variable $x_i$ and $x_j$ are far from each other. Then the decay hypothesis $H_2(q)$ implies that the Gaussian random variable $f_n(x_i)$ and $f_n(x_j)$ are almost independent, and from the definition of the Kac density $\rho_{p,n}$ we deduce that for $A\subset \{1,\ldots,p\}$,
\[\rho_{|A|,n}((x_a)_{a\in A}) \simeq \prod_{I\in\I} \rho_{|A\cap I|,n}((x_a)_{a\in A\cap I}).\numberthis\label{eqter:67}\]
Note that the function $\rho_{p,n}$ depends on $f_n$ only through the covariance matrix of the vector $(f_n(x_1),\ldots,f_n(x_p),f'_n(x_1),\ldots,f'_n(x_p))$. This matrix representation allows us to give a precise error term in \eqref{eqter:67}, proportional to the square of the magnitude of $r_n^{(u,v)}(x_i,x_j)$, where $i$ and $j$ belong to different cells of the partition $\I$. We refer to Section \ref{subsecter2} for matrix notations and to Section \ref{subsecter6} for the matrix representation of the Kac Density.\jump

The combinatoric properties of cumulants and \eqref{eqter:67} imply that
\[F_{p,n}(x_1\ldots,x_p) \simeq 0,\numberthis\label{eqter:03}\]
as soon as the variables $(x_i)_{1\leq i\leq p}$ are clustered with respect to some partition $\I$ with at least two cells. A refinement of Taylor expansion using graph theoretic arguments (see Section \ref{subsecter7}), gives a much more precise error in \eqref{eqter:03} than the approach taken in \cite{Anc21}, where it is roughly shown that a similar approximation as in \eqref{eqter:03} holds true only when one single variable is far from all the others (this reasoning also appears in different articles that treats cumulant asymptotics, see for instance \cite{Naz10, Bla19}). We then show that far from the deep diagonal $(x,\ldots,x)$ the function $F_{p,n}$ is small and will have sufficiently nice integrability properties on $(nU)^p$ in order to show in \eqref{eqter:57} that for $p\geq 3$,
\[\lim_{n\rightarrow +\infty} \frac{1}{n^{p/2}}\int_{(nU)^p} \prod_{i=1}^p\phi\left(\frac{x_i}{n}\right) F_{p,n}(x_1\ldots,x_p)\dd x_1\ldots\dd x_p = 0.\] 
Given the link between cumulants and central moments, this fact leads to the convergence of the central moment of order $p$ to the central moment of a Gaussian random variable. If moreover, the function $g_\omega$ is in $L^{\frac{p}{p-1}}(\R)$ then the function $F_{p,n}(0,x_2,\ldots,x_p)$ is integrable on $(nU)^{p-1}$, uniformly for $n\in\oNN$. This fact leads to the exact asymptotics of the $p$-th cumulant of the random variable $\langle \nu_n, \phi\rangle$.\jump

Despite its apparent simplicity, the detailed proof is quite technical and the diversity of arguments used justifies the following section, which introduces several notions and associated notations for the rest of the paper. In particular, the notion of partition of a finite set plays a central role in this article. From a combinatoric point of view, it appears in the Kac--Rice formula when expressing moments of the factorial power counting measure in terms of moments of the usual power measure, but also from the interpretation of cumulants in the context of Möebius inversion in the lattice of partition. The interplay between these last two combinatoric facts leads to an elegant expression of the cumulants of the zeros counting measure (given by Proposition \ref{Propter02}), and simplifies the approach taken by the authors in \cite{Anc21}, where they computed directly the asymptotics of central moments. \jump

 A novelty of this paper is also the intensive use of the matrix representation of the Kac density which allows us to dissociate the probabilistic setting, and facts concerning pure matrix analysis. We believe that this approach, already taken by the author in \cite{Gas21} to treat the asymptotic of the variance, greatly simplifies the exposition of proofs using Kac--Rice formulas.
\section{Basics and notations}
We define a few notations that will be of use and simplify the exposition. In the following, $A$ is a non-empty finite set. The letter $a,b,\ldots$ denote elements of $A$. The letters $B,C,\ldots$ denote subset of $A$. The letters $\I,\J,\ldots$ denote subsets of the power set of $A$.
\subsection{Partition and cumulants}
\subsubsection{Set theory}
We denote by $|A|$ the cardinal of the set $A$ and $\Pt(A)$ the power set of $A$. For a set $E$, we define $E^A$ the product of $|A|$ copies of $E$. A generic element of $E^A$ is denoted
\[\x_A = (x_a)_{a\in A}\]
to avoid any confusion when elements of $A$ are also sets. For a function $f:E\rightarrow\R$ and $\x_A\in E^A$ we  write
\[f(\x_A) = (f(x_a))_{a\in A}.\]
Let $\uphi_A=(\phi_a)_{a\in A}$ be functions from $E$ to $\R$. We define 
\[\uphi_A^\otimes : \x_A\mapsto \prod_{a\in A}\phi_a(x_a).\numberthis\label{eqter:18}\]
At last, we denote by
\[2A = \{1,2\}\times A.\numberthis\label{eqter:09}\]
The set $2A$ should be seen as the disjoint union of $A$ and a copy of itself. For an element $\x_A\in E^A$ we denote $\x_{A,A}$ the element $(\x_A,\x_A)\in E^{2A}$.
\subsubsection{The lattice of partitions and cumulants}
\label{subsecter4}
The material of this paragraph is very standard, we refer to \cite{Spe83, Pec11} for a nice introduction on this topic. We define $\Pt_A$ as the set of partitions of $A$. The partition of $A$ into singletons will be denoted $\overline{A}$. In the following, $B$ is a subset of $A$ and $\I$ is a partition of $A$. For $a\in A$ we denote $[a]_\I$ the cell of $\I$ in which the element $a$ belongs, and $\I_B$ the partition of $B$ induced by the partition $\I$ of $A$. For instance, if $\I = \{\{1,2\},\{3,4\},\{5\}\}$, $a=1$ and $B = \{1,2,3\}$ then
\[[a]_\I = \{1,2\}\quand \I_B = \{\{1,2\},\{3\}\}.\]
The partition $\I$ induce a partition on the set $2A$ via the relation
\[\enstq{2I}{I\in\I},\numberthis\label{eqter:11}\]
and we will still denote by $\I$ this partition.
\jump

The set $\Pt_A$ has a natural structure of a poset (partially ordered set). Given $\I$ and $\J$ two partition of $A$, we say that $\I$ is \textit{finer} than $\J$ (or that $\J$ is \textit{coarser} than $\I$) and we denote it $\I\preceq \J$ (or $\J\succeq \I$), if 
\[\forall I\in \I,\;\exists J\in \J\;\text{such that}\;I\subset J.\]
We then have
\[\I_J = \enstq{I\in\I}{I\subset J}\quand J = \bigsqcup_{I\in\I_J} I.\]
Note that there is a one-to-one correspondence between the set of partitions of $A$ coarser than a partition $\I$, and the set of partition of $\I$, given by the application
\[\J\mapsto \enstq{\I_J}{J\in\J}.\numberthis\label{eqter:63}\]
Following this observation, we denote $[I]_\J$ the cell of $\J$ in which the set $I$ is included. For instance, if $\I = \{\{1\},\{2,3\},\{4\},\{5\}\}$ and $\J = \{\{1,2,3\}\{4,5\}\}$ then $\I\preceq \J$ and
\[[\{2,3\}]_\J = \{1,2,3\}\quand \I_{\{4,5\}} = \{\{4\},\{5\}\}.\]

Note that two partitions $\I$ and $\J$ have a greatest lower bound and a least upper bound for this partial order, which turns $(\Pt_A,\preceq)$ into a finite lattice. Let $(m_B)_{B\subset A}$ and $(\kappa_B)_{B\subset A}$ be two families of numbers. In our case of interest, the Möebius inversion on this particular lattice takes the form
\[\left(\forall B\subset A,\quad m_B = \sum_{\I\in\Pt_B} \prod_{I\in\I}\kappa_{I}\right)\quad \text{iff}\quad\left(\forall B\subset A,\quad \kappa_B = \sum_{\I\in\Pt_B} (-1)^{|\I|-1}(|\I|-1)!\prod_{I\in\I}m_I\right).\numberthis\label{eqter:62}\]
We will make use of the following cancellation property of the cumulants.
\begin{lemma}
\label{Lemmater52}
Let $(m_B)_{B\subset A}$ and $(\kappa_B)_{B\subset A}$ be two families of numbers related by one of the equivalent formulas in \eqref{eqter:62}. Assume the existence of a partition $\I\neq \{A\}$ such that
\[\forall B\subset A,\quad m_B = \prod_{I\in \I} m_{I\cap B}.\]
Then 
\[\kappa_A=0.\]
\end{lemma}
\begin{proof}
See \cite{Spe83}.
\end{proof}
If $(X_a)_{a\in A}$ is a family of random variables, we can define for a subset $B$ of $A$
\[m_B((X_b)_{b\in B}) = \E\left[\prod_{b\in B}X_b\right]\quand \kappa_B((X_b)_{b\in B}) = \sum_{\I\in\Pt_B} (-1)^{|\I|-1}(|\I|-1)!\prod_{I\in\I}\E\left[\prod_{i\in I}X_i\right].\numberthis\label{eqter:68}\]
The quantity $m_B((X_b)_{b\in B})$ (resp. $\kappa_B((X_b)_{b\in B})$) is the joint moment (resp. cumulant) of the family of random variables $(X_b)_{b\in B}$. The previous Lemma \ref{Lemmater52} translates in the following property for the cumulant. If there is a partition $\I$ with at least two cells, such that the collection of random variables $(X_i)_{i\in I}$ for $I\in\I$ are mutually independent, then the joint cumulant of the family $(X_a)_{a\in A}$ is zero.\jump

The joint cumulants are a convenient tool in the Gaussian framework, since for a Gaussian vector $(X_a)_{a\in A}$, the joint cumulant $\kappa_A((X_a)_{a\in A})$ cancels as soon as $|A|\geq 3$. Conversely, a random variable $X$ such that $\kappa_p(X,\ldots,X) = 0$ for all $p\geq 3$ is Gaussian.
\subsection{Diagonal set and factorial power measure}
\label{subsecter1}
We will see in Section \ref{subsecter3} that the Kac--Rice formula gives an integral formula for the $p$-th factorial power measure of the zero set of a Gaussian process. The expression of the Kac density degenerates near the diagonal and it motivates the introduction of a few notations for the diagonal of a set and factorial power measure. In the following, $A$ is a finite set and $(E,d)$ is a metric space. This section is largely inspired by \cite[Sect. 4.3]{Anc19} and \cite[Sect. 6.1]{Anc21}, in particular for the quick and efficient description of the diagonal clustering.
\subsubsection{Diagonal set and diagonal inclusion}
We define the (large) diagonal of $E^A$ as the subset
\[\Delta := \Delta^A = \enstq{\x_A\in E^A}{\exists a,b\in A\quad\text{with}\quad a\neq b\quand x_a=x_b}.\]
Let $\I$ be a partition of the set $A$. We define 
\[\Delta_{\I} = \enstq{\x_A\in E^A}{x_a=x_b\;\Longleftrightarrow\;[a]_\I = [b]_\I}.\]
From this definition, one has the following decomposition of the space $E$
\[E^A = \bigsqcup_{\I\in\Pt_A}\Delta_\I, \quad\Delta = \bigsqcup_{\substack{\I\in\Pt_A\\\I\neq \overline{A}}}\Delta_\I\quand E^A\setminus\Delta = \Delta_{\overline{A}},\]
where $\overline{A}$ is the partition of $A$ in singletons. We also define
\[\Delta_{\I^+} := \bigsqcup_{\J\preceq \I}\Delta_\J = \enstq{\x_A\in E^A}{x_a=x_b\;\Longrightarrow\;[a]_\I = [b]_\I}.\]
\subsubsection{Enlargement of the diagonal set}
We fix a number $\eta\geq 0$ and $\x_A\in E^A$. We define the graph $G_\eta(\x_A)$ with set vertices $A$, and where two vertices $a$ and $b$ are connected by an edge if $d(x_a,x_b)\leq\eta$. Denote by $\I_\eta(\x_A)$ the partition of $A$ induced by the connected components of $G_\eta(\x_A)$. It allows us to define the subset
\[\Delta_{\I,\eta} = \enstq{\x_A\in E^A}{\I_\eta(\x_A)=\I}. \numberthis\label{eqter:22}\]
If $\eta=0$ then $\Delta_{\I,0}=\Delta_\I$. In the case where $\eta> 0$ we have $\overline{\Delta_{\I,\eta}} \subset \Delta_{\I^+}$. As in the case $\eta=0$, we also have
\[E^A = \bigsqcup_{\I\in\Pt_A}\Delta_{\I,\eta}.\]
The fundamental property of this construction is the following. Let $a,b\in A$ and $\x_A\in \Delta_{\I,\eta}$. If $[a]_\I = [b]_\I$ then 
\[d(x_a,x_b)\leq |A|\eta,\]
and if $[a]_\I \neq [b]_\I$ then 
\[d(x_a,x_b)> \eta.\]
Note that any partition of the space $E^A$ indexed by all the possible partitions of $A$ that satisfies these two properties would also work. The one described above is a quick and efficient way to prove the existence of such partition.
\subsubsection{The factorial power measure}
We define the diagonal inclusion 
\begin{align*}
\iota_\I : E^\I&\longrightarrow E^A\\
\x_\I&\longrightarrow (x_{[a]_\I})_{a\in A}.
\end{align*}
For instance, if $\I = \{\{1,3\},\{2\}\}$ then $\iota_\I (x,y) = (x,y,x)$. A direct consequence of this definition is that the mapping $\iota_\I$ is a bijection between $E^\I\setminus\Delta$ and $\Delta_\I$.\jump

Let $Z$ be a locally finite subset of the metric space $E$. We set $\nu:= \sum_{x\in Z}\delta_x$ the counting measure on $Z$, and
\[\nu^A = \sum_{x\in Z^A}\delta_x\quad\quand\quad\nu^{[A]} = \sum_{x\in Z^A\setminus \Delta}\delta_x.\]
The measure $\nu^A$ (resp. $\nu^{[A]}$) is the power (resp. factorial power) measure of the measure $\nu$. Both measures are linked by the following lemma.
\begin{lemma}
\label{Lemmater01}
With the notations as above, one has
\[\nu^A = \sum_{\I\in\Pt_A}\iota_\I\,\!_*\nu^{[\I]}.\]
\end{lemma}
\begin{proof}
We have 
\[\sum_{x\in Z^A}\delta_x = \sum_{\I\in\Pt_A}\left(\sum_{x\in Z^A\cap\Delta_\I}\delta_x\right).\]
Using the fact that the mapping $\iota_\I$ is a bijection between $E^\I\setminus\Delta$ and $\Delta_\I$, one gets
\[\sum_{x\in Z^A\cap\Delta_\I}\delta_x = \sum_{y\in Z^\I\setminus\Delta}\delta_{\iota_\I(y)}=\iota_\I\,\!_*\nu^{[\I]}.\]
\end{proof}
\subsection{Matrix notations}
\label{subsecter2}
The Kac density (see Section \ref{subsecter3} and Lemma \ref{Lemmater22}) is expressed in term of the covariance matrix of the underlying Gaussian process and its derivatives. This fact allows us to consider the Kac density as a function defined on the set of positive definite matrices, evaluated in some covariance matrix related to our underlying Gaussian process. To this end, we introduce a few useful notations
\subsubsection{Basic matrix notations}
In the following, we define $\M_A(\R)$, $\Ss_A(\R)$ and $\Ss_A^+(\R)$ respectively the sets of square, symmetric and symmetric positive definite matrices acting on the space $\R^A$ equipped with its canonical basis. If $B$ is another finite set, we define $\M_{B,A}(\R)$ the set of matrices from $\R^A$ to $\R^B$. The open subset of matrices in $\M_{B,A}(\R)$ with maximal rank is denoted $\M_{B,A}^*(\R)$. For a matrix $\Gamma \in \M_{B,A}(\R)$ we define
\[\|\Gamma\| = \sup_{i,j} |\Gamma_{i,j}|.\]
Given a matrix $\Sigma \in \M_{2B,2A}(\R)$, we write
\[\Sigma = \begin{pmatrix}[c|c]
\Sigma^{11} & \Sigma^{12}\\
\hline
\Sigma^{21} & \vphantom{\int^{\int^a}}\Sigma^{22}\quad
\end{pmatrix},\quad \Sigma^{11}, \Sigma^{12}, \Sigma^{21}, \Sigma^{22}\in \M_{B,A}(\R).\]
Let $\Sigma\in\M_{2A}(\R)$. If the matrix $\Sigma^{11}$ is invertible, we define the matrix $\Sigma^c\in\M_A(\R)$ to be the Schur complement of $\Sigma^{11}$ in $\Sigma$ :
\[\Sigma^c = \Sigma^{22} - \Sigma^{21}(\Sigma^{11})^{-1}\Sigma^{12}.\]
This matrix arises from the identity
\[\begin{pmatrix}[c|c]
\Id & 0\\
\hline
-\Sigma^{21}(\Sigma^{11})^{-1} & \vphantom{\int^{\int^a}}\;\Id\;\quad
\end{pmatrix}\begin{pmatrix}[c|c]
\Sigma^{11} & \Sigma^{12}\\
\hline
\Sigma^{21} & \vphantom{\int^{\int^a}}\Sigma^{22}\quad
\end{pmatrix} = \begin{pmatrix}[c|c]
\Sigma^{11} & \Sigma^{12}\\
\hline
0 & \vphantom{\int^{\int^a}}\Sigma^{c}\quad
\end{pmatrix}.\]
In particular
\[\det(\Sigma) = \det(\Sigma^{11})\det(\Sigma^c).\numberthis\label{eqter:23}\]
If $\Sigma\in\Ss_{2A}^+(\R)$ then $\Sigma^c\in\Ss_A^+(\R)$ and
\[(\Sigma^c)^{-1} = (\Sigma^{-1})^{22}.\numberthis\label{eqter:24}\]
\subsubsection{Covariance matrix and Gaussian conditioning}
Let $\underline{X}_A = (X_a)_{a\in A}$ and $\underline{Y}_A = (Y_a)_{a\in A}$ two sequences of jointly centered Gaussian vectors. We assume that the Gaussian vector $(\underline{X}_A,\underline{Y}_A)$ is non-degenerate. We define
\[\Sigma^{11} = \Cov(\underline{X}_A),\quad\Sigma^{22} = \Cov(\underline{Y}_A),\quad \Sigma^{12}=\Cov(\underline{X}_A,\underline{Y}_A),\]
and
\[\Sigma := \Cov\left[\left(\underline{X}_A,\underline{Y}_A\right)\right] = \begin{pmatrix}[c|c]
\Sigma^{11} & \Sigma^{12}\\
\hline
^T\Sigma^{12} & \vphantom{\int^{\int^a}}\Sigma^{22}\quad
\end{pmatrix}.\]
\begin{lemma}
\label{Lemmater05}
We have
\[\Law(\underline{Y}_A|\underline{X}_A=0) \sim \Nn(0, \Sigma^c).\]
\end{lemma}
\begin{proof}
We define the Gaussian vector
\[\underline{Y}_A^c = \underline{Y}_A -\,\!^T\Sigma^{12}(\Sigma^{11})^{-1}\underline{X}_A.\]
Then
\[\Cov(\underline{X}_A,\underline{Y}_A^c) = 0\quand\Cov(\underline{Y}_A^c) = \Sigma^c.\]
Since decorrelation implies independence for Gaussian vectors, we have the following equality of conditional distributions
\[\Law(\underline{Y}_A|\underline{X}_A) \sim \Nn(\,\!^T\Sigma^{12}(\Sigma^{11})^{-1}\underline{X}_A, \Sigma^c)\quand\Law(\underline{Y}_A|\underline{X}_A=0) \sim \Nn(0, \Sigma^c).\]
\end{proof}
\subsubsection{Compactness in matrix sets}
The following lemmas give equivalent conditions to being compact in several matrix spaces.
\begin{lemma}
\label{Lemmater02}
A set $K$ is relatively compact in $\Ss_{A}^+(\R)$ if and only if one can find positive constants $c_K$ and $C_K$ such that
\[\forall \Sigma\in K,\quad \det\Sigma\geq c_K,\quad \|\Sigma\|\leq C_K.\]
\end{lemma}
\begin{lemma}
\label{Lemmater03}
Let $B$ be a subset of $A$. A set $K$ is relatively compact in $\M_{B,A}^*(\R)$ if and only if one can find positive constants $c_K$ and $C_K$ such that
\[\forall Q\in K,\quad \det Q\,\!^T\!Q\geq c_K,\quad \|Q\|\leq C_K.\]
\end{lemma}
\begin{lemma}
\label{Lemmater04}
Let $K$ be a relatively compact subset of $\M_{B,A}^*(\R)\times\Ss_{A}^+(\R)$. Then the set
\[\enstq{Q\Sigma\,\!^T\!Q}{(Q,\Sigma)\in K}\]
is relatively compact in $\Ss_B^+(\R)$.
\end{lemma}

\begin{proof}[Proof of the lemmas]
The proof of Lemma \ref{Lemmater02} is a direct consequence of the continuity of the determinant. For Lemma \ref{Lemmater03}, note that a matrix $Q$ of $\M_{B,A}(\R)$ is of maximal rank if and only if $\det Q\,\!^T\!Q > 0$. The conclusion follows again from the continuity of the determinant. As for Lemma \ref{Lemmater04}, let $\Sigma\in \Ss_{A}^+(\R)$ and $Q\in \M_{B,A}^*(\R)$. A direct computation shows that the matrix $Q\Sigma\,\!^T\!Q$ is positive definite. The conclusion then follows from the continuity of the application $(Q,\Sigma)\mapsto Q\Sigma\,\!^T\!Q$.
\end{proof}

\subsubsection{Block diagonal matrix with respect to a partition}
Let $B$ and $C$ be subsets of $A$, and $\Gamma\in \M_{B,C}(\R)$. For $I$ and $J$ subset of $A$ we define
\[\Gamma_{I,J} = (\Gamma_{i,j})_{i\in I\cap B,j\in J\cap C}\quand \Gamma_I = \Gamma_{I,I}.\numberthis\label{eqter:08}\]
Now let $\Sigma\in \M_{2B,2C}(\R)$. We define similarly
\[\Sigma_{I,J} = \begin{pmatrix}[c|c]
\Sigma^{11}_{I,J} & \Sigma^{12}_{I,J}\\
\hline
\Sigma^{21}_{I,J} & \vphantom{\int^{\int^a}}\Sigma^{22}_{I,J}\quad
\end{pmatrix}\quand \Sigma_I  = \Sigma_{I,I}.\]
For a partition $\I$ of the set $A$ and a matrix $\Gamma\in \M_{B,C}(\R)$ we define $\Gamma_{\I}$ to be the block diagonal matrix with blocks $(\Gamma_I)_{I\in\I}$. Similarly, for a matrix $\Sigma\in\M_{2B,2C}(\R)$ we define
\[\Sigma_\I = \begin{pmatrix}[c|c]
\Sigma^{11}_{\I} & \Sigma^{12}_{\I}\\
\hline
\Sigma^{21}_{\I} & \vphantom{\int^{\int^a}}\Sigma^{22}_{\I}\quad
\end{pmatrix}.\]
Note that if $\Sigma \in \Ss_{2A}^+(\R)$ such that $\Sigma = \Sigma_\I$, then 
\[(\Sigma^c)_\I = (\Sigma_\I)^c,\numberthis\label{eqter:25} \]
but the equality is not true in all generality.
\subsubsection{Power product space}
\label{subsecter8}
We introduce a technical matrix space, that will be central in the alternative expression of the cumulant Kac density in Section \ref{subsecter9}. We define the sets
\[\widetilde{\M}_A(\R) := \prod_{B\subset A}\left(\M_B(\R)\times \M_{2B,2A}(\R)\right)\quand \widetilde{\M}_A^*(\R) := \prod_{B\subset A}\left(\M_B(\R)\times \M_{2B,2A}^*(\R)\right).\numberthis\label{eqter:26}\]
The space $\widetilde{\M}_A^*(\R)$ is an open subset of $\widetilde{\M}_A(\R)$. An element $(M^B,Q^B)_{B\subset A}$ of $\widetilde{\M}_A(\R)$ will be denoted $(\widetilde{M},\widetilde{Q})$, with
\[\widetilde{M} = (M^B)_{B\subset A}\quand \widetilde{Q} = (Q^B)_{B\subset A}.\]
If $C$ is a subset of $A$, we denote $(\widetilde{M}_C,\widetilde{Q}_C)$ the element of  $\widetilde{\M}_C(\R)$ defined by
\[\widetilde{M}_C = \left(M^B_C\right)_{B\subset C}\quand \widetilde{Q}_C = \left(Q^B_C\right)_{B\subset C}.\]
At last, if $\I$ is a partition of $A$, we define
\[\widetilde{M}_\I = \left(M^B_\I\right)_{B\subset A}\quand \widetilde{Q}_\I = \left(Q^B_\I\right)_{B\subset A}.\]
\subsection{Divided differences}
\label{subsecter5}
We now introduce the notion of divided differences. Classically used in interpolation theory, we use it in order to give a non-degenerate expression of the Kac density near the diagonal. This approach was first taken in \cite{Cuz75} in order to give a necessary and sufficient condition for the finiteness of the $p$-th moment of the number of zeros on an interval, and has been extensively used in \cite{Anc21}. The results of this section is standard material about interpolation, but we will recall basic properties of divided differences.
\subsubsection{Definition and basic properties}
\paragraph{Definition}\strut\jump

\noindent
Let $f$ be a function defined on an open interval $U$ of $\R$ or $\T$. We use the notations of Section \ref{subsecter1}. In particular, we consider $\Delta$ the large diagonal of $U^A$.
 Let $A$ be a finite set and $\x_A \in U^A\setminus\Delta$. We define $\R^A[X]$ the space of polynomials of degree $|A|-1$. The polynomial
\[L[f;\x_A] :x\mapsto \sum_{a\in A}f(x_a)\prod_{b\neq a} \frac{x-x_b}{x_a-x_b}\]
interpolates the function $f$ at the point $(x_a)_{a\in A}$. It is the only polynomial in $\R^A[X]$ with this property, since the difference of two such polynomials cancels at least $|A|$ times and thus must be zero. The application 
\[\pi_{\x_A} : f\mapsto L[f;\x_A]\]
is a projector onto the space $\R^A[X]$, whose kernel is the space of functions that cancels on $\x_A$. We then define the divided difference of $f$ as the coefficient of degree $|A|-1$ in $L[f;\x_A]$. More explicitly,
\[f[\x_A] = \sum_{a\in A} f(x_a) \prod_{b\neq a} \frac{1}{x_a-x_b}.\]
For instance,
\[f[x] = f(x),\quad f[x,y] = \frac{f(x)-f(y)}{x-y},\quad f[x,y,z] = \frac{f(x)(z-y) + f(y)(x-z) + f(z)(y-x)}{(x-y)(y-z)(z-x)},\]
and so on. The following lemma is analogous of Taylor expansion theorem, in the context of divided differences.
\begin{lemma}
\label{Lemmater06}
For $a\in A$, one has
\[L[f;\x_A](x) = L[f;\x_{A\setminus \lbrace a\rbrace}](x) + f[\x_A]\prod_{b\neq a}(x-x_b),\]
and
\[f(x) = L[f;\x_A](x) + f[\x_A,x]\prod_{a\in A}(x-x_a).\]
\end{lemma}
\begin{proof}
For $a\in A$, the polynomial
\[x\mapsto L[f;\x_A](x) - f[\x_A]\prod_{b\neq a}(x-x_b)\]
interpolates the points $(x_b)_{b\neq a}$ and is of degree $|A|-2$. By uniqueness of the interpolating polynomial, it coincides with the polynomial $L[f;\x_{A\setminus \lbrace a\rbrace}]$. Hence the first statement. An application of this formula with interpolating points $A\cup\lbrace x\rbrace$ yields the second statement.
\end{proof}
%
\paragraph{Continuity property of the divided differences}\strut\jump

\noindent
Assume from now on that the function $f$ is of class $\CC^{|A|-1}$. Define $C_A$ to be the standard simplex of dimension $|A|-1$ :
\[C_A = \enstq{\underline{t}_A\in [0,1]^A}{\sum_{a\in A}t_a=1}.\]
We can equip the simplex $C_A$ with the induced rescaled Lebesgue measure $\dd m$, so that $m(C_A)= \frac{1}{(|A|-1)!}$. We then have the following integral representation for the divided differences, which is analogous to the integral remainder in Taylor expansion in the context of divided differences.
\begin{lemma}[Hermite--Genocchi formula]
\label{Lemmater07}
We have
\[f[\x_A] = \int_{C_A} f^{(|A|-1)}\left(\sum_{a\in A}t_ax_a\right)\dd m(\underline{t}_A).\]
In particular, the application $\x_A\mapsto f[\x_A]$ continuously extends to the whole space $U^A$.
\end{lemma}
\begin{proof}
See \cite{Boo05}.
\end{proof}

This proposition allows us to extend by continuity the functions 
\[\x_A\mapsto  f[\x_A]\quand\x_A\mapsto L[f;\x_A],\]
from $U^A\setminus \Delta$ to the whole space $U^A$. For instance, if $x_a = y$ for all $a\in A$, then 
\[f[\x_A] = \frac{f^{(|A|-1)}(y)}{(|A|-1)!}\quand L[f;\x_A](x) = \sum_{i=0}^{|A|-1} \frac{f^{(i)}(y)}{i!}(x-y)^i.\numberthis\label{eqter:12}\]
This expression coincides with the Taylor polynomial of order $|A|-1$ of the function $f$ at the point $y$. The continuity property of the divided differences allows us to extend Lemma \ref{Lemmater06} and \ref{Lemmater07} by taking $\x_A$ in the whole space $U^A$ and $x\in \R$.\jump
\paragraph{Divided difference of a polynomial}\strut\jump

\noindent
In this section, $P$ denotes a polynomial. The Hermite--Genocchi  formula \ref{Lemmater07} directly implies the following lemma.
\begin{lemma}
\label{Lemmater08}
The quantity $P[\x_A]$ is a polynomial expression of the coefficients $\x_A$.
\end{lemma}
For instance, if $P(x) = x^4$ then
\[P[x,y] = x^3 + x^2y + xy^2 + y^3.\]
If $\deg P<|A|$ then $L[P;\x_A,x] = P$, and the coefficient of degree $|A|$ of this polynomial is zero, which implies that $P[\x_A,x]=0$. In the following we assume that $\deg P\geq |A|$.
\begin{lemma}
\label{Lemmater09}
The polynomial $x\mapsto P[\x_A,x]$ is of degree $\deg P-|A|$, and the leading coefficient this polynomial is the leading coefficient of the polynomial $P$.
\end{lemma}
\begin{proof}
From the definition of the divided differences, 
\[P[\x_A,x] \underset{x\rightarrow+\infty}{\simeq} \frac{P(x)}{\prod_{a\in A}(x-x_a)},\]
which implies that the polynomial $x\mapsto P[\x_A,x]$ is of degree $\deg P-|A|$, and its leading coefficient is the leading coefficient of $P$.
\end{proof}

\paragraph{Iterated divided differences}\strut\jump

\noindent
Let $B$ be a subset of $A$ and $\x_B\in U^B$. We define the function
\[f^{[\x_B]}:x\mapsto f[\x_B,x].\numberthis\label{eqter:16}\]
\begin{lemma}
\label{Lemmater10}
Let $\x_A\in U^A$. Then
\[L[f;\x_A][\x_B,\,.\,] = L[f^{[\x_B]},\x_{A\setminus B}],\]
and
\[f[\x_A] = f^{[\x_B]}[\x_{A\setminus B}].\]
Let $x\in U$. Then
\[f[\x_B,x] = L[f;\x_A][\x_B,x] + f[\x_A,x]\prod_{a\in A\setminus B}(x-x_a).\]
\end{lemma}
\begin{proof}
Consider the two polynomials
\[P_1 = L[f^{[\x_B]};\x_{A\setminus B}]\quad \quand\quad P_2 = L[f;\x_A]^{[\x_B]}.\]
They both interpolate the points $\x_{A\setminus B}$ at values $(f[\x_B,x_a])_{a\in A\setminus B}$. The polynomial $P_1$ is in $\R^{A\setminus B}[X]$, and $L[f;\x_A]$ is in $\R^A[X]$. By Lemma \ref{Lemmater09}, the polynomial $P_2$ is also in $\R^{A\setminus B}[X]$. By uniqueness of the interpolating polynomial, these two polynomials are equal, hence the first statement.\jump

The coefficient of degree $|A\setminus B|-1$ in $P_1$ is $f^{[\x_B]}[\x_{A\setminus B}]$, and the one in $P_2$ is $f[\x_A]$ according to Lemma \ref{Lemmater09}. By the previous equality these two coefficients are equal, which yields the second formula. The last formula is a direct application of Lemma \ref{Lemmater06} applied to the function $f^{[\x_B]}$.
\end{proof}
\subsubsection{Matrix viewpoint of the divided differences}
In order to describe the divided differences of a Gaussian process and the induced transformation on the covariance matrix, we rewrite the operation of taking divided differences from a matrix viewpoint. From now on, we equip $A$ with an arbitrary total order $\leq$ and we introduce the notation
\[a\mid A = \enstq{b\in A}{b\leq a}.\]
Thus, $\x_{\,a\mid A} = (x_b)_{b\leq a}$.
\paragraph{Basis of polynomials adapted to the divided difference}\strut\jump

\noindent
For $\x_A\in \R^A$ we define the polynomial
\[P^a_{\x_A}: x\mapsto \prod_{b<a} (x-x_b).\]
For any subset $B$ of $A$, the family $(P^b_{\x_B})_{b\in B}$ is a basis of the space $\R^B[X]$ and we will always equip the space $\R^B[X]$ with this basis. 
\begin{remark}
\label{Remarkter1}
The family $(P^a_{\x_A})_{a\in A}$ is a family of monic polynomials of increasing degrees. Thus, if we equip $A$ with another total order, the underlying transformation matrix is of determinant $1$ and the coefficients are polynomial quantities in $(x_i-x_j)_{i,j\in A}$. It justifies in the following why the order can be chosen arbitrarily. 
\end{remark}
A direct induction based on Lemma \ref{Lemmater06} implies the following lemma.
\begin{lemma}
\label{Lemmater11}
Let $\x_A\in U^A$. Then
\[L[f;\x_A] = \sum_{a\in A} f[\x_{\,a\mid A}]P^a_{\x_A}.\]
\end{lemma}

The finite differences of $f$ are thus the coefficients of the Lagrange interpolation polynomial in the basis $(P^a_{\x_A})_{a\in A}$. We define 
\[f_A[\x_A] = (f[\x_{a|A}])_{a\in A}\quand M(\x_A) = (P^b_{\x_A}(x_a))_{a,b\in A}.\]
Then Lemma \ref{Lemmater11} rewrites matricially
\[f(\x_A) = M(\x_A)f_A[\x_A]\numberthis\label{eqter:05}.\]
For instance,
\[\begin{pmatrix}
f(x) \\ 
f(y) \\
f(z)
\end{pmatrix}  = \begin{pmatrix}
1 & 0 & 0\\ 
1 & y-x & 0\\
1 & z-x & (z-x)(z-y)
\end{pmatrix}\begin{pmatrix}
f[x] \\ 
f[x,y]\\
f[x,y,z]
\end{pmatrix} .\]
The matrix $M(\x_A)$ is lower triangular, thus
\[\det M(\x_A) = \prod_{a\in A} P^a_{\x_A}(x_a) = \prod_{a<b}(x_b-x_a).\numberthis\label{eqter:04}\]
\paragraph{Divided differences with respect to a partition}\strut\jump

\noindent
In the following, $\I$ is a partition of the set $A$. We define
\[f_\I[\x_A] = (f_I[\x_I])_{I\in \I} = (f[\x_{i|I}])_{i\in I,I\in\I}\]
We can perform the divided difference independently on each cell of the partition. That is, we can write for all $I\in\I$
\[f(\x_I) = M(\x_I)f_I[\x_I]\quand f(\x_A) = M^\I(\x_A)f_\I[\x_A];\numberthis\label{eqter:06}\]
where $M^\I(\x_A)$ is the block diagonal matrix with blocks $(M(\x_I))_{I\in\I}$. For instance, if $\I = \{\{1,2\},\{3,4\}\}$ then
\[\begin{pmatrix}
f(w) \\
f(x) \\ 
f(y) \\
f(z)
\end{pmatrix} = \begin{pmatrix}
1 & 0 & 0 & 0 \\ 
1 & x-w & 0 & 0 \\ 
0 & 0 & 1 & 0 \\ 
0 & 0 & 1 & z-y
\end{pmatrix} \begin{pmatrix}
f[w] \\
f[w,x] \\ 
f[y] \\
f[y,z]
\end{pmatrix}.
\]
From Equation \eqref{eqter:04}, 
\[\det M^\I(\x_A) = \prod_{I\in\I}\det M(\x_I) = \prod_{I\in\I}\prod_{\substack{i,j\in I\\i< j}}(x_j-x_i).\]

\begin{lemma}
\label{Lemmater12}
Let $\x_A\in \Delta_{\I,\eta}$. There is a constant $C(\eta)$ such that
\[\|M^\I(\x_A)\| \leq C(\eta).\]
\end{lemma}
\begin{proof}
Let $a,b\in A$. If $[a]_\I\neq[b]_\I$, then $(M^\I(\x_A))_{a,b} = 0$. Else, there is $I\in\I$ such that $a,b\in I$. Then
\[(M^\I(\x_A))_{a,b} = P^b_{\x_I}(x_a) = \prod_{\substack{c<b\\c\in I}} (x_a-x_c).\]
For $a,c\in I$, one has $|x_a-x_c|\leq |A|\eta$. The conclusion follows.
\end{proof}
For $\x_A\in U^A$ we consider the mapping
\begin{align*}
\pi_{\x_A}^\I : \R^A[X]&\longrightarrow \prod_{I\in\I}\R^I[X]\\
P&\longrightarrow (L[P;\x_I])_{I\in\I}.
\end{align*}
It is well defined for $\x_A\in  U^A$, since a polynomial is infinitely differentiable. For a subset $I$ of $A$, we equip $\R^I[X]$ with the basis of polynomials $(P^i_{\x_I})_{i\in I}$. Let $Q^\I(\x_A)$ be the matrix of the application $\pi_{\x_A}^\I$ in that basis. For instance, when $\I = \{\{1,2\},\{3\}\}$ then
\[\begin{pmatrix}
f[x] \\ 
f[x,y] \\
f[z]
\end{pmatrix}  = \begin{pmatrix}
1 & 0 & 0\\ 
0 & 1 & 0\\
1 & z-x & (z-x)(z-y)
\end{pmatrix}\begin{pmatrix}
f[x] \\ 
f[x,y]\\
f[x,y,z]
\end{pmatrix}.\]
Let $\x_A\in U^A\setminus\Delta$. Then $M^\I(\x_A)$ is invertible and from Equation \eqref{eqter:05} and \eqref{eqter:06}, one has
\[f_\I[\x_A] = [M^\I(\x_A)]^{-1}M(\x_A)f_A[\x_A],\]
and thus
\[Q^\I(\x_A) = [M^\I(\x_A)]^{-1}M(\x_A).\]
It implies that
\[\left|\det Q^\I(\x_A)\right| = \prod_{\substack{\{I,J\} \subset \I\\I \neq J}}\prod_{i\in I}\prod_{j\in J}|x_i-x_j|.\numberthis\label{eqter:07}\]
By continuity with respect to $\x_A$ of the application $\pi_{\x_A}^\I$ this formula remain true for $\x_A\in U^A$. We deduce that the application $\pi_{\x_A}^\I$ is invertible for $\x_A\in\Delta_{\I^+}$. \jump
Now let $\J$ be another partition of $A$, finer than the partition $\I$. For $\x_A\in 
U^A$ we consider the mapping
\begin{align*}
\pi_{\x_A}^{\I,\J} : \prod_{I\in\I}\R^I[X]&\longrightarrow \prod_{J\in\J}\R^J[X]\\
(P_I)_{I\in\I} &\longrightarrow (L[P_{[J]_I};\x_J])_{J\in\J}.
\end{align*}
Restricted to $\R^I[X]$, the application $\pi_{\x_A}^{\I,\J}$ coincides with $\pi_{\x_I}^{\J_I}$. Let $Q^{\I,\J}(\x_A)$ be the matrix of the application $\pi_{\x_A}^{\I,\J}$, so that
\[f_\J[\x_A]= Q^{\I,\J}(\x_A)f_\I[\x_A].\numberthis\label{eqter:13}\]
\begin{lemma}
\label{Lemmater13}
\[\left|\det Q^{\I,\J}(\x_A))\right| = \prod_{I\in \I} \prod_{\substack{\{ J_1,J_2\}\subset \J_I\\J_1 \neq J_2}}\prod_{j_1\in J_1}\prod_{j_2\in J_2}|x_{j_1}-x_{j_2}|.\]
In particular, the application $\pi_{\x_A}^{\I,\J}$ is invertible in the case where $\x_A\in \Delta_{\J^+}$.
\end{lemma}
\begin{proof}
We have from Equation \eqref{eqter:07}
\[\left|\det Q^{\I,\J}(\x_A))\right| = \prod_{I\in\I} \left|\det Q^{\J_I}(\x_I))\right|= \prod_{I\in \I} \prod_{\substack{\{ J_1,J_2\}\subset \J_I\\J_1 \neq J_2}}\prod_{j_1\in J_1}\prod_{j_2\in J_2}|x_{j_1}-x_{j_2}|.\]
And this expression does not cancel when $\x_A\in \Delta_{\J^+}$.
\end{proof}
\paragraph{Divided difference on a subset}\strut\jump

\noindent
In this section, $B$ is a subset of $A$ and $\I$ is a partition of $A$. From Equation \eqref{eqter:05} and \eqref{eqter:06}, one has
\[f(\x_B) = M_{B,A}(\x_A)f_A[\x_A]\quand f(\x_B) = M_{B,A}^\I(\x_A)f_\I[\x_A],\]
where for a matrix $M$, the matrix $M_{B,A}$ is defined in \eqref{eqter:08}. For $\x_A\in U^A$ we consider the mapping
\begin{align*}
\pi_{\x_A}^{\I,B} : \prod_{I\in\I}\R^I[X]&\longrightarrow \prod_{I\in\I}\R^{I\cap B}[X]\\
(P_I)_{I\in\I} &\longrightarrow (L[P_I;\x_{I\cap B}])_{I\in\I},\numberthis\label{eqter:10}
\end{align*}
and let $Q^{\I,B}(\x_A)$ be the matrix of this application.
\begin{lemma}
\label{Lemmater14}
The set
\[\enstq{Q^{\I,B}(\x_A)}{\x_A\in \Delta_{\I,\eta}}\]
is relatively compact in $\M_{B,A}^*(\R)$.
\end{lemma}
\begin{proof}
Let $\x_A\in\Delta_{\I,\eta}$. The matrix $Q^{\I,B}(\x_A)$ is block diagonal with respect to the partition $\I$, with blocks $(Q^{\{I\},B\cap I}(\x_I))_{I\in\I}$. If $B\cap I = b|I$ for some $b\in B\cap I$, then
\[Q^{\{I\},B\cap I}(\x_I) = \begin{pmatrix}[c|c]
\Id_{B\cap I}\vphantom{\int^\int} &
0\vphantom{\int^\int}
\end{pmatrix},\]
and the conclusion follows. If not, we change the order on $I$, so that $B\cap I = b|I$ for some $b\in B\cap I$. According to Remark \ref{Remarkter1}, the change-of-basis matrix is of determinant $1$ and the coefficients are polynomial expressions in $(x_i-x_j)_{i,j\in I}$. Since $\x_A\in \Delta_{\I,\eta}$, there are bounded, and the matrix $Q^{\{I\},B\cap I}(\x_I)$ is of maximal rank. The conclusion in the general case follows from Lemma \ref{Lemmater04}.

\end{proof}
\paragraph{Doubling divided differences}\strut\jump

\noindent
In the paragraph, we consider the divided difference on the set $2A$ defined in \eqref{eqter:09}. We equip the set $2A$ with the lexicographic order inherited from the order on $\{1,2\}$ and the arbitrary order on $A$. Note that
\[(f[x],f[x,x]) = (f(x),f'(x)).\]
The interest of doubling divided differences is to consider simultaneous interpolation of the function $f$ and $f'$ on prescribed points $(\x_a)_{a\in A}$. This part coincides with the classical Hermite interpolation and we will assume that $f$ is of class $\CC^{2|A|-1}$.\jump

Let $\I$ be a partition of $A$ (and thus of $2A$, following the notation \eqref{eqter:11}). As a consequence of Definition \eqref{eqter:10}, one has
\[f_{\I_B}[\x_{B,B}] = Q^{\I,2B}(\x_{A,A})f_{\I}[\x_{A,A}].\numberthis\label{eqter:48}\]
We define
\[\widetilde{M}^\I(\x_A) = (M^{\I_B}(\x_B))_{B\subset A}\quand \widetilde{Q}^\I(\x_A) = (Q^{\I,2B}(\x_{A,A}))_{B\subset A}. \numberthis\label{eqter:27}\]
Recall the definition of $\widetilde{\M}_{A}(\R)$ and $\widetilde{\M}_{A}^*(\R)$ in Paragraph \ref{subsecter8}. One has the following key proposition.
\begin{proposition}
\label{Propter01}
Let $\x_A\in \R^A$. Then
\[(\widetilde{M}^\I(\x_A),\widetilde{Q}^\I(\x_A))\in\widetilde{\M}_{A}(\R),\]
\[\widetilde{M}^\I(\x_A) = \widetilde{M}_\I^\I(\x_A)\quand \widetilde{Q}^\I(\x_A) = \widetilde{Q}^\I_\I(\x_A),\]
\[\forall B\subset A,\;\;\forall I\in\I,\quad M^{\I_B}_I(\x_B) = M(\x_{I\cap B})\quand  Q^{\I,2B}_I(\x_{A,A}) = Q^{\I,2(B\cap I)}(\x_{A,A}).\]
Moreover, the set
\[\enstq{(\widetilde{M}^\I(\x_A),\widetilde{Q}^\I(\x_A))}{\x_A\in \Delta_{\I,\eta}}\]
is relatively compact in $\widetilde{\M}^*_{A}(\R)$.
\end{proposition}
\begin{proof}
The first three assertions directly follow from the definition of $\widetilde{M}^\I$ and $\widetilde{Q}^\I$. As for the second proposition, let $\x_A\in\Delta_{\I,\eta}$. According to Lemma \ref{Lemmater12}, the coefficients of the matrix $M^{\I_B}(\x_B)$ are bounded by a constant $C(\eta)$. We also have $\x_{A,A}\in \Delta_{\I,\eta}$. Lemma \ref{Lemmater14} applies, and the set of matrices $(Q^{\I,2B}(\x_{A,A}))_{\x_A\in \Delta_{\I,\eta}}$ is relatively compact in $\M_{2B,2A}^*(\R)$. The conclusion follows.
\end{proof}
\paragraph{Divided differences of a Gaussian process}\strut\jump

\noindent
At last we describe the covariance matrix of the divided difference vector of a Gaussian process. The integral representation given by the Hermite--Genocchi formula gives a convenient expression for the coefficients of the covariance matrix.\jump

Let $f$ be a Gaussian process of class $\CC^{|A|-1}$ on the interval $U$. We denote $r$ the covariance function of $f$, which is differentiable $|A|-1$ times in each variable. 
We define
\[\Sigma^\I(\x_A) = \Cov(f_\I[\x_A]).\numberthis\label{eqter:66}\]
\begin{lemma}
\label{Lemmater16}
Let $I,J\in \I$, $a\in I$ and $b\in J$. Then
\[\Sigma^\I(\x_A)_{ab} = \int_{C_{a|I}\times C_{b|J}}r^{(|\,a|I\,|-1\,,\,|\,b|J\,|-1)}\left(\sum_{i\in a|I}s_ix_i,\sum_{j\in b|J}t_jx_j\right)\dd m(\underline{s}_{i|I})\dd m(\underline{t}_{i|J}).\]
\end{lemma}
\begin{proof}
The Hermite Genocchi formula \ref{Lemmater07} asserts that
\begin{align*}
\Sigma^\I(\x_A)_{ab} &= \E[f[\x_{a|I}]f[\x_{b|J}]]\\
&= \int_{C_{a|I}\times C_{b|J}}\E\left[f^{(|\,a|I\,|-1)}\left(\sum_{i\in a|I}s_ix_i\right)f^{(|\,b|J\,|-1)}\left(\sum_{j\in b|J}t_jx_j\right)\right]\dd m(\underline{s}_{a|I})\dd m(\underline{t}_{b|J})\\
&= \int_{C_{a|I}\times C_{b|J}}r^{(|\,a|I\,|-1,|\,b|J\,|-1)}\left(\sum_{i\in a|I}s_ix_i,\sum_{j\in b|J}t_jx_j\right)\dd m(\underline{s}_{a|I})\dd m(\underline{t}_{b|J}).
\end{align*}
\end{proof}
\section{Kac--Rice formula for Gaussian processes}
\label{subsecter3}
For now on, $A$ is a finite set and $f$ is a centered Gaussian process defined on an interval $U$ of $\R$ or $\T$, with covariance function $r$. We assume for this section that the process $f$ is of class $\CC^{2|A|-1}$, and for every partition $\I$ of $A$ and $\y_\I \in U^{\I}\setminus \Delta$, the following non-degeneracy condition holds
\[\det\left[\Cov\left(\left(f^{(k)}(y_I)\right)_{\substack{0\leq k\leq 2|I|-1\\I\in \I}}\right)\right]>0.\numberthis\label{eqter:20}\]
It has been shown for instance in \cite[Thm. 3.6]{Aza09} that this condition ensures the finiteness of the $|A|$-th moment of the number of zeros of a Gaussian process on a bounded interval.
\subsection{Kac density and cumulants of the zeros counting measure}
In this section we give the expression of the $p$-th factorial moment and cumulant of the zeros counting measure. The first step is to lift the degeneracy of the Kac--Rice formula near the diagonal.
\subsubsection{Non-degeneracy of the Kac density near the diagonal}
We apply the method of divided differences to lift the degeneracy of the Kac density and give an alternative formula near the diagonal. The material is this section is quite standard, see for instance \cite{Aza09, Anc21}. Only Lemma \ref{Lemmater19} is new and allows us to express the Kac density as a function of a non-degenerate Gaussian vector. \jump

We fix for the rest of this paragraph a partition $\I$ of $A$.
\begin{lemma}
\label{Lemmater17}
The Gaussian vectors $f_\I[\x_A]$ and $f_\I[\x_{A,A}]$ are non-degenerate for $\x_A\in\Delta_{\I^+}$.
\end{lemma}
\begin{proof}
We prove first the non-degeneracy of $f_\I[\x_A]$. The process $f$ is of class $\CC^{|A|-1}$, and the Gaussian vector $f_\I[\x_A]$ is well-defined. Let $\x_A\in\Delta_{\I^+}$. By definition of the set $\Delta_{\I^+}$, there is a partition $\J$ finer than the partition $\I$ and such that $\x_A\in \Delta_\J$. We write $\x_A = \iota_\J(\y_\J)$ for some $\y_\J\in\R^\J\setminus\Delta$. For $J\in\J$, we have from equation \eqref{eqter:12}
\[f_\J[\x_A] = \left(\frac{f^{(k)}(y_J)}{k!}\right)_{\substack{0\leq k\leq |J|-1\\J\in \J}},\]
which is non-degenerate by the hypothesis \eqref{eqter:20}. Now from Equation \eqref{eqter:13},
\[f_\J[\x_A]= Q^{\I,\J}(\x_A)f_\I[\x_A].\]
According to Lemma \ref{Lemmater13}, the matrix $Q^{\I,\J}(\x_A)$ is invertible when $\x_A\in\Delta_\J$, which implies that the Gaussian vector $f_\I[\x_A]$ is non-degenerate for $\x_A\in\Delta_{\I^+}$.\jump

Now for the Gaussian vector $f_\I[\x_{A,A}]$, the process $f$ is of class $\CC^{2|A|-1}$ and the Gaussian vector $f_\I[\x_{A,A}]$ is well-defined. Moreover, if $\x_A\in\Delta_{\I^+}$ then $\x_{A,A}\in \Delta_{\I^+}$ and we can apply the previous case to the set $2A$ to deduce the non-degeneracy of the Gaussian vector $f_\I[\x_{A,A}]$.
\end{proof}

We define the random set
\[Z = \enstq{x\in U}{f(x)=0}.\]
By Bulinskaya lemma (see \cite[Lem. 1.20]{Aza09}) and the assumption on $f$, the subset $Z$ is almost surely a closed discrete subset of $U$ and we can define the random measure $\nu$ to be the counting measure of $Z$. The Kac--Rice formula (see \cite[Thm.~3.2]{Aza09} and \cite[Prop. 3.6]{Anc21}) asserts that for a measurable function $\Phi:U^A\rightarrow \R$ bounded with compact support, one has, following the notations of Section \ref{subsecter1},
\[\E[\langle \nu^{[A]},\Phi\rangle] = \int_{U^A} \Phi(\x_A)\rho(\x_A)\dd \x_A \numberthis\label{eqter:28},\]
with
\[\rho(\x_A):= \rho_{|A|}(\x_A) = \frac{N(\x_A)}{D(\x_A)},\]
where
\[N(\x_A) = \E\left[\prod_{a\in A}|f'(x_a)|\;\middle|\; \forall a\in A,\; f(x_a)=0\right]\quand D(\x_A) = \sqrt{\det\left[2\pi\Cov(f(x_A))\right]}.\]
The function $\rho$ is only defined for $\x_A\in U^A\setminus\Delta$. Along the diagonal $\Delta$, the function $N$ is ill-defined and the function $D$ cancels. The first step consists in giving an alternative non singular formula for $\rho$ in a neighborhood of the diagonal $\Delta$. 
Let $\x_A\in \Delta_{\I^+}$. We define
\[D^\I(\x_A) = \sqrt{\det[2\pi\Cov(f_\I[\x_A])]},\quad N^\I(\x_A) = \E\left[\prod_{I\in\I}\prod_{i\in I}\left|f[\x_I,x_i]\right| 
\;\middle|\; 
f_\I[\x_A]=0\right],\]
and
\[\rho^\I(\x_A) = \frac{N^\I(\x_A)}{D^\I(\x_A)}.\]
Lemma \ref{Lemmater17} implies that the three quantities above are well defined.
\begin{remark}
\label{Remarkter3}
If $\I= \overline{A}$, then $f_{\overline{A}}(\x_A) = f(\x_A)$ and $f[x_a,x_a] = f'(x_a)$, which implies that 
\[\rho(\x_A) = \rho^{\overline{A}}(\x_A).\]
\end{remark}
One has the following relations.
\begin{lemma}
\label{Lemmater18}
Let $\J$ be a finer partition than $\I$. Then for $\x_A\in \Delta_{\J^+}$ one has
\begin{align*}
D^\J(\x_A) &= \left|\det Q^{\I,\J}(\x_A)\right|D^\I(\x_A),\\
N^\J(\x_A) &= \left(\det Q^{\I,\J}(\x_A)\right)^2N^\I(\x_A),\\
\rho^\J(\x_A) &= \left|\det Q^{\I,\J}(\x_A)\right|\rho^\I(\x_A).
\end{align*}
\end{lemma}

It implies that the function $\rho^\J$ can be extended by continuity from $\Delta_{\J^+}$ to $\Delta_{\I^+}$ via these relations. By taking $\I = \{A\}$ and $\J = \overline{A}$, it implies that the function $\rho$ can be extended by continuity to the whole space $U^A$ and is bounded. Moreover, one has
\[\det Q^{\{A\},\overline{A}}(\x_A) = \det M(\x_A) = \prod_{a<b} (x_b-x_a),\]
thus the function $\rho$ vanishes on the diagonal $\Delta$.
\begin{proof}
Let $\x_A\in \Delta_{\J^+}$. From Equation \eqref{eqter:13},
\[f_\J[\x_A] = Q^{\I,\J}(\x_A)f_\I[\x_A].\]
We deduce that
\begin{align*}
D^\J(\x_A) &= \sqrt{\det\left[2\pi\Cov(f_\J[\x_A])\right]}\\
&=\left|\det Q^{\I,\J}(\x_A)\right|D^\I(\x_A).\numberthis\label{eqter:14}
\end{align*}
The quantities $N_A^\I(\x_A)$ and $N_A^\J(\x_A)$ are well defined. The Gaussian vectors $f_\I[\x_A]$ and $f_\J[\x_A]$ are equal up to a linear invertible change of variable, and they cancel simultaneously. In other words, one has
\begin{align*}
N^\J(\x_A) &= \E\left[\prod_{J\in\J}\prod_{j\in J}\left|f[\x_J,x_j]\right| 
\;\middle|\; 
f_\J[\x_A]=0\right]\\
&= \E\left[\prod_{J\in\J}\prod_{j\in J}\left|f[\x_J,x_j]\right| 
\;\middle|\; 
f_\I[\x_A]=0\right]\\
&= \E\left[\prod_{I\in\I}\prod_{J\in\J_I}\prod_{j\in J}\left|f[\x_J,x_j]\right| 
\;\middle|\; 
f_\I[\x_A]=0\right].
\end{align*}
Let $I\in\I$, $J\in\J_I$ and $j\in J$. From Lemmas \ref{Lemmater10} and \ref{Lemmater11}, conditionally on $f_I[\x_I] = 0$ one has
\[f[\x_J,x_j] = \left(\prod_{\substack{i\in I\setminus J}} (x_j-x_i)\right)f[\x_I,x_j],\]
from which we deduce
\begin{align*}
N^\J(\x_A) &= \left(\prod_{I\in\I}\prod_{J\in \J_I}\prod_{j\in J}\prod_{i\in I\setminus J} |x_i-x_j|\right)\E\left[\prod_{I\in\I}\prod_{i\in I}|f[\x_I,x_i]|\;\middle|\; f_\I[\x_A]=0\right]\\
&= \left(\prod_{I\in\I}\prod_{\substack{J_1,J_2\subset \J_I\\J_1 \neq J_2}}\prod_{j_1\in J_1}\prod_{j_2\in J_2}|x_{j_1}-x_{j_2}|\right)N^\I(\x_A).\numberthis\label{eqter:15}
\end{align*}
We deduce the alternative expression for $\rho^\J$ from \eqref{eqter:14} and \eqref{eqter:15}.
\end{proof}
When the points $(\x_a)_{a\in A}$ collapse on the diagonal $\Delta_\I$ the vector $(f[\x_I,\x_i])_{I\in\I,i\in I}$ becomes degenerate, which makes unpractical the analysis of regularity of the function $\rho^\I$ in a neighborhood of the diagonal $\Delta_\I$. The following lemma gives another expression of the quantity $N^\I(\x_A)$ that depends fully on a non-degenerate Gaussian vector. Recall the definition of the function $f^{[\x_B]}$ for a subset $B$ of $A$ in \eqref{eqter:16}.
\begin{lemma}
\label{Lemmater19}
One has
\[N^\I(\x_A) = \E\left[\prod_{I\in\I}\prod_{i\in I}\left|\left(M(\x_I)f^{[\x_I]}_I[\x_I]\right)_i\vphantom{\int}\right| 
\;\middle|\; 
f_\I[\x_A]=0\right].\]
\end{lemma}
\begin{proof}
Let $I\in\I$. According to formula \eqref{eqter:05},
\[f^{[\x_I]}(\x_I) = M(\x_I)f^{[\x_I]}_I[\x_I].\]
The conclusion follows from the definition of $N^\I(\x_A)$.
\end{proof}
\subsubsection{Expression of the cumulants of the zeros counting measure}
We are now ready to give the expression of the cumulant of order $|A|$ of the linear statistics associated to zeros counting measure. Let $(\phi_a)_{a\in A}$ be a collection of bounded functions with compact support. We define
\[\kappa_A(\nu)(\underline{\phi}_A) = \kappa\left(\left(\langle \nu,\phi_a \rangle\right)_{a\in A}\right),\]
the joint cumulant of the family of random variables $(\langle \nu,\phi_a \rangle)_{a\in A}$. We define the \textit{cumulant Kac density} associated with the set $A$ to be the function
\begin{align*}
F_A : \R^A&\longrightarrow \R\\
\x_A &\longrightarrow \sum_{\J\in \Pt_A}(|\J|-1)!(-1)^{|\J|-1}\prod_{J\in\J}\rho(\x_J).\numberthis\label{eqter:50}
\end{align*}
The following Proposition \ref{Propter02} express the cumulant of order $|A|$ of the linear statistics associated to zeros counting measure. It is key step in towards proof of Theorem \ref{Theoremter9}, and reveals the elegant interplay between the factorial power counting measure and the combinatorics of cumulants. The formula is quite standard in the study of $k$-point function of point processes, see for instance \cite[Claim 4.3]{Naz10}.
\begin{proposition}
\label{Propter02}
We have
\[\kappa_A(\nu)(\underline{\phi}_A) = \sum_{\I\in\Pt_A} \int_{U^\I}\left(\prod_{I\in\I}\prod_{i\in I} \phi_i(x_I)\right)F_\I(\x_\I)\dd\x_\I.\]
\end{proposition}
\begin{proof}
We have, using the expression of cumulants in terms of moments given by \eqref{eqter:68} and the notation \eqref{eqter:18}
\[\kappa_A(\nu)(\underline{\phi}_A) = \sum_{\I\in\Pt_A}(|\I|-1)!(-1)^{|\I|-1}\prod_{I\in\I}\E\left[\left\langle \nu^I,\underline{\phi}_I^\otimes \right\rangle\right].\]
The link between the power measure and factorial power measure given by Lemma \ref{Lemmater01} implies that
\[\E\left[\left\langle \nu^I,\underline{\phi}_I^\otimes \right\rangle\right] = \sum_{\J\in\Pt_I}\E\left[\left\langle \nu^{[\J]},\underline{\phi}_I^\otimes\circ \iota_\J \right\rangle\right].\]
The bijection given by \eqref{eqter:63} then implies
\begin{align*}
\kappa_A(\nu)(\underline{\phi}_A) &= \sum_{\I\in\Pt_A}(|\I|-1)!(-1)^{|\I|-1}\sum_{\J\preceq \I}\prod_{I\in\I}\E\left[\left\langle \nu^{[\J_I]},\underline{\phi}_I^\otimes\circ \iota_{\J_I} \right\rangle\right]\\
&= \sum_{\J\in \Pt_A}\sum_{\I\succeq \J}(|\I|-1)!(-1)^{|\I|-1}\prod_{I\in\I}\E\left[\left\langle \nu^{[\J_I]},\underline{\phi}_I^\otimes\circ \iota_{\J_I} \right\rangle\right].
\end{align*}
The Kac--Rice formula then asserts that
\[\E\left[\left\langle \nu^{[\J_I]},\underline{\phi}_I^\otimes\circ \iota_{\J_I} \right\rangle\right] = \int_{U^{\J_I}}\left(\prod_{\substack{J\in\J\\J\subset I}}\prod_{j\in J}\phi_j(x_J)\right)\rho(\x_{\J_I})\dd\x_{\J_I},\]
from which we deduce that
\begin{align*}
\kappa_A(\nu)(\underline{\phi}_A) &= \sum_{\J\in \Pt_A}\int_{U^\J}\left(\prod_{J\in\J}\prod_{j\in J} \phi_j(x_J)\right)\sum_{\I\succeq \J}(|\I|-1)!(-1)^{|\I|-1}\prod_{I\in\I}\rho(\x_{\J_I})\dd\x_{\J_I}\\
&= \sum_{\J\in \Pt_A}\int_{U^\J}\left(\prod_{J\in\J}\prod_{j\in J} \phi_j(x_J)\right)F_\J(\x_\J)\dd\x_\J,
\end{align*}
where the last equality follows from the bijection given by \eqref{eqter:63}.
\end{proof}
For instance if $|A| = 2$ then the second order cumulant coincides with the variance and
\[\kappa_A(\nu)(\underline{\phi}_A) = \left(\int_{U^2} \phi_1(x)\phi_2(y)\left[\rho_2(x,y)-\rho_1(x)\rho_1(y)\right]\dd x\dd y\right) + \left(\int_U\phi_1(x)\phi_2(x)\rho_1(x)\dd x\right).\]
\subsection{Matrix representation of the Kac density and factorization property}
\label{subsecter6}
In this section we prove a matrix representation for the Kac density and the cumulant Kac density. It allows us to dissociate the analysis of the covariance matrix of divided differences associated with the Gaussian process $f$, and of the Kac density seen as a functional of the covariance matrix.
\subsubsection{Matrix representation of the Kac density}
We define the mapping
\begin{align*}
\widetilde{\rho}:\M_A(\R)\times\Ss_{2A}^+(\R)\,&\longrightarrow \R\\
M\times\Sigma\quad&\longmapsto \frac{1}{\sqrt{\det(2\pi\Sigma)}}\int_{\R^A}\prod_{a\in A}\left|(M\y_A)_a\right| \exp\left(-\frac{^T\!\y_A(\Sigma^c)^{-1}\y_A}{2}\right)\dd \y_A.
\end{align*}
Recall from definition \eqref{eqter:66} that
\[\Sigma^\I(\x_A) = \Cov(f_\I[\x_A])\quand \Sigma^\I(\x_{A,A}) = \Cov(f_\I[\x_{A,A}]).\]
The following lemma gives an alternative expression of $\rho$ as a function of the covariance matrix $\Sigma^\I(\x_{A,A})$, and the matrix of divided differences $M^\I(\x_A)$ defined in \eqref{eqter:06}.
\begin{lemma}
\label{Lemmater22}
For $\x_A\in\Delta_{\I^+}$,
\[\rho(\x_A) = \left|\det M^\I(\x_A)\right|\widetilde{\rho}\left(M^\I(\x_A),\Sigma^\I(\x_{A,A})\right).\]
\end{lemma}
\begin{proof}
Note first that Remark \ref{Remarkter3} and Lemma \ref{Lemmater18} implies that
\[\rho(\x_A) = \left|\det Q^{\I,\overline{A}}(\x_A)\right|\rho^\I(\x_A) = \left|\det M^\I(\x_A)\right|\rho^\I(\x_A).\]
Let $I\in\I$. In virtue of Lemma \ref{Lemmater10}, one has 
\[f^{[\x_I]}_I[\x_I] = (f[\x_I,x_{i|I}])_{i\in I}\quad\text{and thus}\quad \left(f_I[\x_I],f^{[\x_I]}_I[\x_I]\right) = f_{2I}[\x_{I,I}].\]
Following the notations of Section \ref{subsecter2} it implies that
\[\Sigma^\I(\x_{A,A})^{11} = \Cov(f_\I[\x_A]).\]
From Equation \eqref{eqter:23}, one has
\[\det \left[2\pi\Sigma^\I(\x_{A,A})\right]=\det\left[2\pi\Sigma^\I(\x_{A,A})^{11}\right]\det\left[2\pi\Sigma^\I(\x_{A,A})^c\right].\]
Using the alternative expression of $N^\I$ given by Lemma \ref{Lemmater19} and the conditional formula of Lemma \ref{Lemmater05}, we deduce
\begin{align*}
\rho^\I(\x_A) &= \frac{1}{\sqrt{\det \left[2\pi\Sigma^\I(\x_{A,A})^{11}\right]}}\E\left[\prod_{I\in\I}\prod_{i\in I}\left|\left(M(\x_I)f^{[\x_I]}_I[\x_I]\right)_i\vphantom{\int}\right| 
\;\middle|\; 
f_\I[\x_A]=0\right]\\
&= \frac{1}{\sqrt{\det\left[2\pi\Sigma^\I(\x_{A,A})\right]}}\int_{\R^A}\prod_{a\in A}\left|(M^\I(\x_A)\y_A)_a\right| \exp\left(-\frac{^T\!\y_A(\Sigma^\I(\x_{A,A})^c)^{-1}\y_A}{2}\right)\dd \y_A\\
&= \widetilde{\rho}\left(M^\I(\x_A), \Sigma^\I(\x_{A,A})\right).
\end{align*}
The conclusion follows.
\end{proof}
\begin{lemma}
\label{Lemmater23}
Let $\x_A\in\Delta_{\I^+}$. For a subset $B$ of $A$,
\[\rho(\x_B) = \left|\det M^{\I_B}(\x_B)\right|\widetilde{\rho}\left(M^{\I_B}(\x_B),Q^{\I,2B}(\x_{A,A})\Sigma^\I(\x_{A,A})\,\!^T\!Q^{\I,2B}(\x_{A,A})\right).\]
\end{lemma}
\begin{proof}
Recall that the partition $\I_B$ of $B$ is the partition induced by the partition $\I$ on the subset $B$. If $\x_A\in\Delta_{\I^+}$ then $\x_B\in\Delta_{\I_B^+}$. We can thus apply the previous Lemma \ref{Lemmater22} to get
\[\rho(\x_B) = \left|\det M^{\I_B}(\x_B)\right|\widetilde{\rho}\left(M^{\I_B}(\x_B),\Sigma^{\I_B}[\x_B]\right).\]
From Equation \eqref{eqter:48},
\[f_{\I_B}[\x_{B,B}] = Q^{\I,2B}(\x_{A,A})f_\I[\x_{A,A}],\]
thus
\[\Sigma^{\I_B}[\x_B] = Q^{\I,2B}(\x_{A,A})\Sigma^\I(\x_{A,A})\,\!^T\!Q^{\I,2B}(\x_{A,A}).\]
\end{proof}

Given two open subsets $\Omega_1$ and $\Omega_2$ of finite dimensional vector spaces, we define the function space $\CC^{0,\infty}(\Omega_1,\Omega_2)$ to be the set of functions from $\Omega_1\times\Omega_2$ to $\R$, that are infinitely differentiable with respect to the second argument and such that the partial derivatives (with respect to the second argument) are jointly continuous. We endow this space with the usual topology of uniform convergence of second partial derivatives to any order on compact subsets of $\Omega_1\times\Omega_2$.  
\begin{lemma}
\label{Lemmater24}
The application $\widetilde{\rho}$ belongs to $\CC^{0,\infty}(\M_A(\R),\Ss_{2A}^+(\R))$.
\end{lemma}
\begin{proof}
Let
\[h(\Sigma,\y_A) =  \frac{1}{\sqrt{\det(2\pi\Sigma)}}\exp\left(-\frac{^T\!\y_A(\Sigma^c)^{-1}\y_A}{2}\right).\]
The function $\Sigma\mapsto h(\Sigma,\y_A)$ is infinitely differentiable on $\Ss_{2A}^+(\R)$ and its partial derivatives are also exponentially decreasing with respect to the variable $\y_A$. By differentiability under the integral, it implies that $\widetilde{\rho}$ belongs to $\CC^{0,\infty}(\M_A(\R),\Ss_{2A}^+(\R))$.
\end{proof}
\subsubsection{Factorization of the Kac density and error term}
In this section, we show that the function $\widetilde{\rho}$ satisfies a nice factorization property. This is a rigorous statement of the approximation \eqref{eqter:67} stated in introduction. For the rest of this section, $\I$ is a fixed partition of the set $A$. 
\begin{lemma}
\label{Lemmater25}
Let $M\in \M_A(\R)$ and $\Sigma\in \Ss_{2A}^+(\R)$ such that $M = M_\I$ and $\Sigma = \Sigma_\I$. Then
\[\widetilde{\rho}(M,\Sigma)=\prod_{I\in\I}\widetilde{\rho}(M_I,\Sigma_I).\]
\end{lemma}
\begin{proof}
Since the matrix $\Sigma$ is block diagonal with respect to the partition $\I$,
\[\sqrt{\det \left(2\pi\Sigma\right)} = \prod_{I\in\I} \sqrt{\det \left(2\pi\Sigma_I\right)}.\]
Similarly, for $\y_A\in\R^A$,
\[\exp\left(-\frac{^T\!\y_A(\Sigma^c)^{-1}\y_A}{2}\right) = \prod_{I\in \I} \exp\left(-\frac{^T\!\y_I((\Sigma_I)^c)^{-1}\y_I}{2}\right).\]
The matrix $M_\I$ is also block diagonal with respect to the partition $\I$ and
\[\prod_{a\in A}|(M\y_A)_a| = \prod_{I\in\I}\prod_{i\in I}|(M_I\y_I)_i|.\]
The conclusion follows from the definition of $\widetilde{\rho}$.
\end{proof}
We now want to describe the error term in Lemma \ref{Lemmater25} after perturbation of the block-diagonal matrix $\Sigma$. We start with the following lemma.
\begin{lemma}
\label{Lemmater26}
Let $K$ be a compact subset of $\M_A(\R)\times \Ss_{2A}^+(\R)$. There is a constant $C_K$ such that for all $(M,\Sigma)\in K$ with $M=M_\I$, one has
\[\left|\widetilde{\rho}(M,\Sigma) - \widetilde{\rho}(M,\Sigma_\I)\right| \leq C_K \|\Sigma - \Sigma_\I\|^2.\]
\end{lemma}
\begin{proof}
We set $H = \Sigma - \Sigma_\I$. The matrix $H$ is symmetric and satisfies $H_\I = 0$. It implies that
\[(\Sigma^{-1}H)_\I = (\Sigma^{-1}H\Sigma^{-1})_\I = 0\quad\text{and thus}\quad\Tr(\Sigma^{-1}H) = 0.\numberthis\label{eqter:29}\]
One has from identity \eqref{eqter:24}
\begin{align*}
[(\Sigma + H)^c]^{-1} &= [(\Sigma + H)^{-1}]^{22}\\
&= (\Sigma^c)^{-1} - H_\Sigma + O(\|H\|^2),
\end{align*}
where $H_\Sigma = [\Sigma^{-1}H\Sigma^{-1}]^{22}$ and the big-oh is uniform on the compact $K$. By \eqref{eqter:29}, one has $(H_\Sigma)_\I = 0$. Differentiation under the integral sign gives
\[\widetilde{\rho}(M,\Sigma+H) =  \widetilde{\rho}(M,\Sigma) + \dd\widetilde{\rho}_{(M,\Sigma)}.H + O(\|H\|^2),\]
where
\begin{align*}
\dd\widetilde{\rho}_{(M,\Sigma)}.H &= -\frac{1}{2\sqrt{\det\left(2\pi\Sigma\right)}}\int_{\R^A}\prod_{a\in A}\left|(M\y_A)_a\right| \left(^T\!\y_AH_\Sigma\y_A\right)\exp\left(-\frac{^T\!\y_A(\Sigma^c)^{-1}\y_A}{2}\right)\dd \y_A\\
&= -\sum_{\substack{i,j\in A\\ [i]_\I\neq[j]_\I}} \frac{(H_\Sigma)_{ij}}{2\sqrt{\det\left(2\pi\Sigma\right)}}\int_{\R^A}\prod_{a\in A}\left|(M\y_A)_a\right| y_iy_j\exp\left(-\frac{^T\!\y_A(\Sigma^c)^{-1}\y_A}{2}\right)\dd \y_A.
\end{align*}
For each $i,j$ of the sum we make the change of variable 
\[\forall a\in A,\quad z_a = \left\lbrace\begin{array}{ll}
-y_a\; \text{if}\; [a]_\I = [i]_\I\\
\;\;\; y_a\; \text{if}\; [a]_\I \neq [i]_\I
\end{array}\right..\]
Since $M$ and $\Sigma^c$ are block diagonal matrices , one has for $a\in A$ and $\y_A\in\R^A$
\[ |(M\z_A)_a| = |(M\y_A)_a|,\quad ^T\!\z_A(\Sigma^c)^{-1}\z_A = ^T\!\y_A(\Sigma^c)^{-1}\y_A\quad\text{but}\quad z_iz_j = -y_iy_j.\]
Thus
\[ \dd\widetilde{\rho}_{(M,\Sigma)}.H = - \dd\widetilde{\rho}_{(M,\Sigma)}.H = 0,\]
and the conclusion follows.
\end{proof}

We can now state the following proposition that gives the error in Lemma \ref{Lemmater25} when the matrix $\Sigma$ is not block diagonal with respect to the partition $\I$. Note that the following Lemma gives a quadratic error in the matrix coefficients of $\Sigma$, where in the somehow analogous Proposition \cite[prop. 6.43]{Anc21} only proves a square root error. The difference resides in Lemma \ref{Lemmater19}, which allows us to bypass the lack of regularity of Gaussian integrals near the boundary of the cone of symmetric definite matrices.
\begin{corollary}
\label{Corter1}
Let $B$ be a subset of $A$ and $K$ be a compact subset of $\M_B(\R)\times\M_{2B,2A}^*(\R)\times\Ss_{2A}^+(\R)$. Then there is a constant $C_K$ such that, for all $(M,Q,\Sigma) \in K$ such that $M=M_\I$ and $Q=Q_\I$, one has
\[\left|\widetilde{\rho}(M,Q\Sigma\,\!^T\!Q)-\prod_{I\in\I}\widetilde{\rho}(M_I,Q_I\Sigma_I\,\!^T\!Q_I)\right|\leq C_K \sup_{\substack{I,J\in \I\\I\neq J}}\|\Sigma_{I,J}\|^2.\]
\end{corollary}
\begin{proof}
Let $\Pi = Q\Sigma\,\!^T\!Q$. Lemma \ref{Lemmater04} asserts that the couple $(M,\Pi)$ lives in a compact set of $\M_B(\R)\times\Ss_{2B}^+(\R)$.
From Lemma \ref{Lemmater26}, one has
\[\left|\widetilde{\rho}(M,\Pi)- \widetilde{\rho}(M,\Pi_\I)\right| = O(\|\Pi-\Pi_\I\|^2).\]
By Lemma \ref{Lemmater24}, the application $\widetilde{\rho}$ belongs to $\CC^{0,\infty}(\M_B(\R)\times\Ss_{2B}^+(\R))$. Lagrange remainder formula asserts the existence of a constant $C_K$ such that
\[\left|\widetilde{\rho}(M,\Pi) - \widetilde{\rho}(M,\Pi_\I)\right| \leq C_K \|\Pi-\Pi_\I\|^2\leq C_K\sup_{\substack{I,J\in \I\\I\neq J}}\|\Pi_{I,J}\|^2.\]
Since $Q=Q_\I$ and $\|Q\|\leq C_K$ for some constant $C_K$, we deduce
\[\|\Pi_{I,J}\| = \|Q_I\Sigma_{I,J}\,\!^T\!Q_J\|\leq C_K \|\Sigma_{I,J}\|.\]
Finally the conclusion follows from Lemma \ref{Lemmater25} since
\[\widetilde{\rho}(M,\Pi_\I) = \prod_{I\in\I}\widetilde{\rho}(M_I,\Pi_I) = \prod_{I\in\I}\widetilde{\rho}(M_I,Q_I\Sigma_I\,\!^T\!Q_I).\]
\end{proof}
\subsubsection{Matrix representation of the cumulant Kac density}
\label{subsecter9}
Similarly to the Kac density, we can derive a matrix representation for the cumulant Kac density defined in \eqref{eqter:50}. Note that the divided differences do not behave well when we consider them on a subfamily of interpolations points $(\x_A)_{a\in A}$. It explains why we introduced in Paragraph \ref{subsecter8} the somehow complicated set $\widetilde{\M}_A^*(\R)$. We introduce the function $\widetilde{F}_A$ defined by
\begin{align*}
&\widetilde{F}_A:\widetilde{\M}_A^*(\R)\times\Ss_{2A}^+(\R)\longrightarrow \R\\
&(\widetilde{M},\widetilde{Q})\times\Sigma\longrightarrow \sum_{\J\in \Pt_A}(|\J|-1)!(-1)^{|\J|-1}\prod_{J\in\J}\left|\det M^J\right|\widetilde{\rho}\,(M^J,Q^J\Sigma\,\!^TQ^J).
\end{align*}
Let $\I$ be a partition of $A$. The following lemma gives an alternative expression to the function $F_A$ when $\x_A\in\Delta_{\I^+}$.
\begin{lemma}
\label{Lemmater30}
For $\x_A\in\Delta_{\I^+}$ one has
\[F_A(\x_A) = \widetilde{F}_A\left((\widetilde{M}^\I(\x_A),\widetilde{Q}^\I(\x_A)),\Sigma^\I(\x_{A,A})\right).\]
\end{lemma}
\begin{proof}
One has
\[F_A(\x_A) = \sum_{\J\in \Pt_A}(|\J|-1)!(-1)^{|\J|-1}\prod_{J\in\J}\rho(\x_J).\]
According to lemma \ref{Lemmater23}, for a subset $J$ of $A$ one has
\[\rho(\x_J) = \left|\det M^{\I_J}(\x_J)\right|\widetilde{\rho}\left(M^{\I_J}(\x_J),Q^{\I,2J}(\x_{A,A})\Sigma^\I(\x_{A,A})\,\!^T\!Q^{\I,2J}(\x_{A,A})\right).\]
The first statement follows from the definition \eqref{eqter:27} of $\widetilde{M}^\I(\x_A)$ and $\widetilde{Q}^\I(\x_A)$.
\end{proof}
Given the definition of the function $\widetilde{F}_A$, one can translate the cancellation property of Lemma \ref{Lemmater52} to this function. It is the object of the following lemma.
\begin{lemma}
\label{Lemmater15}
Let $\I$ be a partition of $A$, with $\I\neq \{A\}$. Let $(\widetilde{M},\widetilde{Q})\in \widetilde{\M}_A(\R)$ and $\Sigma\in \Ss_{2A}(\R)$ such that $\widetilde{M} = \widetilde{M}_\I$, $\widetilde{Q} = \widetilde{Q}_\I$, $\Sigma = \Sigma_\I$ and
\[\forall B\subset A,\;\;\forall I\in\I,\quad M^B_I = M^{I\cap B}\quand  Q^B_I = Q^{I\cap B}.\]
Then
\[\widetilde{F}_A((\widetilde{M},\widetilde{Q}),\Sigma)=0.\]
\end{lemma}
\begin{proof}
For a subset $B$ be a subset of $A$, we set
\[m_B = |\det M^B|\widetilde{\rho}(M^B,Q^B\Sigma^TQ^B).\]
Then from Lemma \ref{Lemmater25} one has
\begin{align*}
m_B &= |\det M^B|\;\widetilde{\rho}(M^B,Q^B\Sigma_I^TQ^B)\\
&= \prod_{I\in\I}|\det M^{I\cap B}|\;\widetilde{\rho}(M^{I\cap B},Q^{I\cap B}\Sigma^TQ^{I\cap B})\\
&= \prod_{I\in\I} m_{I\cap B}
\end{align*}
From Lemma \ref{Lemmater52}, one deduce that
\begin{align*}
\widetilde{F}_A((\widetilde{M},\widetilde{Q}),\Sigma) = \sum_{\J\in \Pt_A}(|\J|-1)!(-1)^{|\J|-1}\prod_{J\in\J}m_J = 0.
\end{align*}

\end{proof}
Corollary \ref{Corter1} translates directly into the following bound for the function $\widetilde{F}_A$. In the following, $K$ is a compact subset of $\widetilde{\M}_A^*(\R)\times\Ss_{2A}^+(\R)$.
\begin{lemma}
\label{Lemmater33}
Let $\I$ be a partition of $A$, with $\I\neq \{A\}$. Then there is a constant $C_K$ such that for all $((\widetilde{M},\widetilde{Q}),\Sigma) \in K$ with $\widetilde{M} = \widetilde{M}_\I\quad \widetilde{Q} = \widetilde{Q}_\I$, and
\[\forall B\subset A,\;\;\forall I\in\I,\quad M^B_I = M^{I\cap B}\quand  Q^B_I = Q^{I\cap B},\]
one has
\[\left|\widetilde{F}_A((\widetilde{M},\widetilde{Q}),\Sigma)\right|\leq C_K\sup_{\substack{I,J\in \I\\I\neq J}}\|\Sigma_{I,J}\|^2.\]
\end{lemma}
\begin{proof}
From Lemma \ref{Lemmater15}, one has
\[\widetilde{F}_A((\widetilde{M},\widetilde{Q}),\Sigma_\I) = 0.\]
Since the function $\widetilde{F}_A$ is a polynomial expression in the functions $\widetilde{\rho}$, the error term given by Corollary \ref{Corter1} translates directly for the function $\widetilde{F}_A$ to the desired estimate.
\end{proof}
\subsection{Decay of the cumulant Kac density}
\label{subsecter7}
The goal of the following section is to improve the quadratic bound given by Lemma \ref{Lemmater33}. We will do so, thanks to a refinement of Taylor expansion for regular functions that cancel on given affine subspaces. The next key Lemma \ref{Lemmater40} bounds the function $\widetilde{F}_A$ by a sum over a collection of graphs. We recall first a few definitions and propositions from graph theory.
\subsubsection{Graph setting}
A graph $G$ is a couple $(E(G),V(G))$, where $E(G)$ is the set of vertices of the graph $G$ and $V(G)$ the collection of edges of $G$. For our purposes, a graph $G$ has no loop, but two different edges can have the same endpoints. The multiplicity of an edge $\{a,b\}$ is the number of edges in the graph that are equal to $\{a,b\}$. We say that a graph $G$ is \textit{$2$-edge connected} if the multiplicity of any edge is at most two, and the graph $G$ remains connected when any one of its edges is removed. We define $\G_A$ to be the set of $2$-edge connected graphs with set of vertices $A$. Notice that this set has finite cardinal.\jump

Let $\I$ be a partition of $A$ and let $G$ be a graph with set of vertices $A$. We define the graph $G_\I$ on the set of vertices $\I$ to be the quotient graph with respect to the partition $\I$. That is, the multiplicity of the edge $\{I,J\}$ (with $I\neq J$) of $G_\I$ is the number of edges $\{i,j\}$ in $G$ (with multiplicity) such that $\{I,J\} = \{[i]_\I,[j]_\I\}$.
\begin{lemma}
\label{Lemmater20}
Let $H\in\G_\I$. There is $G\in\G_A$ such that
\[H = G_\I.\]
\end{lemma}
\begin{proof}
For each $I\in\I$, we replace in $H$ the vertex $I$ by the cycle $(i)_{i\in I}$. The neighbors of $I$ are arbitrary linked to vertices of this cycle. The obtained graph with set of vertices $A$ satisfies $G_\I = H$ and is $2$-edge connected.
\end{proof}

An \textit{ear} of a graph $G$ is a path in $G$ such that its internal vertices all have degree two. Note that a cycle is a particular instance of an ear. An \textit{ear decomposition} of the graph $G$ is a union $(P_1,\ldots,P_k)$ such that $P_1$ is a cycle, and for $i\geq 2$, $P_i$ is an ear such that its endpoints belong to $\cup_{j<i} P_i$. We states the following standard fact for $2$-edge connected graphs (see \cite{Whi32}). The proof is a simple induction on the number of ears.
\begin{lemma}
A $2$-edge connected graph admits an ear decomposition. The number of ears is necessary the circuit rank of the graph $G$. Moreover, the starting cycle can be chosen arbitrarily among the cycles of $G$.
\end{lemma}
It implies the following lemma.
\begin{lemma}
\label{Lemmater50}
Let $G$ be a $2$-edge connected graph. There is a family $(T_a)_{a\in A}$ of spanning trees of $G$ such that for every edge $e\in E(G)$, one can find an element $a_e\in A$ such that $e$ is not an edge of the spanning tree $T_{a_e}$.
\end{lemma}
\begin{proof}
Let $P_1$ be a largest cycle in $G$, with vertices $B$, and $(P_1,\ldots,P_k)$ be an ear decomposition of $G$. For $i\geq 1$, we define $E_i$ the set of edges of the path $P_i$. One has $|B|\geq |E_i|$, so that one can find a surjection $\tau_i:B\twoheadrightarrow E_i$.\jump

For $a\notin B$, we define $T_a$ to be any spanning tree of $G$. For $a\in B$, we define $T_a$ to be the graph $G$ where we removed, in each path $E_i$, the edge $\tau_i(a)$. The number $k$ is the circuit rank of the graph $G$. By construction, the graph $T_a$ is connected, so it must be a spanning tree of the graph $G$. Every edge $e\in E_i$ is the image of some element $a_e\in B$ by the surjection $\tau_i$, so that the edge $e$ does not belong to the tree $T_{a_e}$.
\end{proof}
\subsubsection{Crossed Taylor formula}
In this paragraph we prove an enhancement of the Taylor remainder estimates for regular functions that cancel on affine subspaces. A simple observation of this phenomenon is the following. Assume that $F(x,y)$ is a regular function such that in a neighborhood of zero,
\[|F(x,y)|\leq |x|\quand |F(x,y)|\leq |y|.\]
Then for some constant $C$, one has in a neighborhood of zero that
\[|F(x,y)|\leq C|xy|,\]
which improves by a square factor the trivial bound $\sqrt{|xy|}$. We wish to extend this observation to the more complicated function $\widetilde{F}_A$ that satisfies the bounds given by Lemma \ref{Lemmater33} for several partitions $\I$ of $A$. In the following, we give a general statement for this phenomenon.\jump

Let $\Omega$ be an open subset of a finite dimensional vector space $V\simeq \R^E$, $F$ an infinitely differentiable function on $\Omega$, and $\y_E$ be a vector in $V$. We fix an integer $d\in\N$. The following lemma states equivalent conditions for a regular function $F$ to cancel on an affine subspace with order of cancellation at least $d$.
\begin{lemma}
\label{Lemmater38}
Let $B$ be a subset of $E$. Then the three following conditions are equivalent.
\begin{enumerate}
\item For every compact subset $K$ of $\Omega$, there is a constant $C_K$ such that,
\[\forall \x_E\in K,\quad|F(\x_E)|\leq C_K\left(\sup_{b\in B} |x_b-y_b|\right)^d.\]
\item For every multi-index $\n_B\in \N^B$ with $|\n_B|=d$, there exists a function $H_{\n_B}\in \CC^\infty(\Omega)$ such that
\[\forall \x_E\in \Omega,\quad F(\x_E) = \sum_{|\n_B|=d}(\x_B-\y_B)^{\n_B}\;H_{\n_B}(\x_E).\]
\item For all $\w_E\in \Omega$ such that $\w_B = \y_B$, for every multi-index $\m_B\in \N^B$ with $|\m_B|<d$,
\[\partial^{\m_E} F(\w_E) = 0.\]
\end{enumerate}
\end{lemma}

\begin{proof} We can assume that $\Omega$ is a product of open intervals. The general case follows by a partition of unity. We denote by $\Omega_B$ the projection of $\Omega$ to $\R^B$.
\begin{itemize}
\item[$\bullet$] $(2)\Rightarrow (1)$ follows from the boundedness of the functions $H_{\n_B}$ on compact subsets $K$ of $\Omega$.
\item[$\bullet$] $(1)\Rightarrow (3)$ is a consequence of the uniqueness of the polynomial approximation given by Taylor expansion.
\item[$\bullet$] $(3)\Rightarrow (2)$, we distinguish two cases. If $\y_B\in\Omega_B$, then the implication a direct consequence of Taylor expansion with integral remainder of the function $F$ on the segment between points $\x_E$ and $(\x_{E\setminus B}, \y_B)$. If $\y_B\notin\Omega_B$, then there is an index $b\in B$ such that $y_b\notin \Omega_{\{b\}}$. We can then define
\[H(\x_E) = \frac{F(\x_E)}{(x_b-y_b)^d},\quad\text{so that}\quad F(\x_E) = (x_b-y_b)^dH(\x_E).\]
\end{itemize}
\end{proof}
Now we extend the previous Lemma \ref{Lemmater38} to a collection $\B$ of (not necessarily disjoints) subsets of $E$. For a fix positive integer $d$ we define
\[\mathcal{C}_\B = \enstq{\n_E\in \{0,\ldots,d\}^E}{\forall B\in \B,\; |\n_B|\geq d}.\numberthis\label{eqter:35}\]
For instance, if $E=\{1,2,3\}$, $\B = \{\{1,2\},\{2,3\},\{1,3\}\}$ and $d = 2$, then
\[\CC_\B  = \left\{(2,2,0),(0,2,2),(2,0,2),(1,1,1),\;\ldots\;\right\}.\]
\begin{proposition}
\label{Propter04}
Assume that for every $B\in\B$, the function $F$ satisfies the equivalent statements of Lemma \ref{Lemmater38}. Then there exists finitely many non-zero functions $(H_{\n_E})_{\n_E\in \mathcal{C}_\B}\in \CC^\infty(\Omega)$ and such that
\[F(\x_E) = \sum_{\n_E\in\mathcal{C}_\B}(\x_E-\y_E)^{\n_E}H_{\n_E}(\x_E).\]
\end{proposition}
\begin{proof}
Once again, we can assume that the $\Omega$ is a product of open intervals, and for a subset $B$ of $E$ we denote by $\Omega_B$ the projection of $\Omega$ to $\R^B$. The proof is a induction on the size of the set $\B$. If $\B = \lbrace B\rbrace$, this exactly the hypothesis on $F$ (second characterization in Lemma \ref{Lemmater38}). Now let $B\in \B$ and suppose that the lemma is true for the family $\B\setminus\lbrace B\rbrace$. We have 
\[F(\x_E) = \sum_{\n_E\in\mathcal{C}_{\B\setminus \lbrace B\rbrace}}(\x_E-\y_E)^{\n_E}H_{\n_E}(\x_E).\]
As in the proof of \ref{Lemmater38}, we distinguish two cases. Assume first that $\y_B\in\Omega_B$, and let $\w_E\in \Omega$ such that $\w_B = \y_B$. For every multi-index $\m_B\in \N^B$ with $|\m_B|<d$,
\begin{align*}
\partial^{\m_B} F(\w_E) &= \sum_{\n_E\in\mathcal{C}_{\B\setminus \lbrace B\rbrace}}\partial^{\m_B}\left((\, .\,-\y_E)^{\n_E}H_{\n_E}(\, .\,)\right)(\w_E)\\ &= \sum_{\substack{\n_E\in\mathcal{C}_{\B\setminus \lbrace B\rbrace}\\ \n_B\leq \m_B}}(\w_{E\setminus B}-\y_{E\setminus B})^{\n_{E\setminus B}}\;\frac{\m_B!}{(\m_B-\n_B)!}\;\partial^{(\m_B-\n_B)}H_{\n_E}(\w_E)\\
&= 0,
\end{align*}
according to the third characterization in Lemma \ref{Lemmater38}. Let $\w_E = (\x_{E\setminus B},\y_B)$. On cannot directly use Lemma \ref{Lemmater38} to the functions $H_{\n_E}$ because it is not guaranteed that they satisfy one of the equivalent propositions, but it will be the case if we subtract to $H_{\n_E}$ its Taylor expansion. For $\n_E\in \mathcal{C}_{\B\setminus\{B\}}$ and $\x_E\in\Omega$ we define the quantity 
\[\widetilde{H}_{\n_E}(\x_E) = H_{\n_E}(\x_E)\;-\sum_{|\m_B|<d-|\n_B|} \frac{(\x_B-\y_B)^{\m_B}}{(\m_B)!}\;\partial^{(\m_B)}H_{\n_E}(\w_E).\]
If $|\n_B|\geq d$, then $H_{\n_E} = \widetilde{H}_{\n_E}$ and $\n_E\in\mathcal{C}_\B$. If $|\n_B|<d$, then by Taylor expansion with integral remainder (or directly by $(3)\Rightarrow (2)$ of Lemma \ref{Lemmater38}), there exists functions $(H_{\n_E\,,\,\p_B})_{|p_B|=d-|n_B|}$ such that
\[\widetilde{H}_{\n_E}(\x_E) = \sum_{|p_B|=d-|n_B|} (\x_B-\y_B)^{\p_B}H_{\n_E\,,\,\p_B}(\x_E)\numberthis\label{eqter:01}.\] 
Now we compute
\begin{align*}
F(\x_E) &= F(\x_E) - \sum_{|\m_B|<d} \frac{(\x_B-\y_B)^{\m_B}}{\m_B!}\;\partial^{\m_B} F(\w_E)\\
&= \sum_{\n_E\in\mathcal{C}_{\B\setminus \lbrace B\rbrace}}(\x_E-\y_E)^{\n_E}\left(H_{\n_E}(\x_E)\;-\sum_{\substack{|\m_B|<d\\ \m_B\geq \n_B}} \frac{(\x_B-\y_B)^{\m_B-\n_B}}{(\m_B-\n_B)!}\;\partial^{(\m_B-\n_B)}H_{\n_E}(\w_E)\right)\\
&= \sum_{\n_E\in\mathcal{C}_{\B\setminus \lbrace B\rbrace}}(\x_E-\y_E)^{\n_E}\widetilde{H}_{\n_E}(\x_E)\\
&= \sum_{\n_E\in(\mathcal{C}_{\B\setminus \lbrace B\rbrace}\cap \,\mathcal{C}_B)}(\x_E-\y_E)^{\n_E}\widetilde{H}_{\n_E}(\x_E) \\
&\qquad+\qquad\sum_{\n_E\in(\mathcal{C}_{\B\setminus \lbrace B\rbrace}\setminus \,\mathcal{C}_B)}\sum_{|p_B|=d-|n_B|} (\x_{E\setminus B}-\y_{E\setminus B})^{\n_{E\setminus B}}(\x_B-\y_B)^{\n_B+\p_B}H_{\n_E\,,\,\p_B}(\x_E)
\end{align*}
One then have $|\n_B + \p_B|\geq d$ and thus the multi-index $(\n_{E\setminus B},\n_B + \p_B)$ belongs to $\mathcal{C}_\B$. The conclusion follows.\jump

If $\y_B\notin\Omega_B$, we can argue as in the proof of $(3)\Rightarrow (2)$ in Lemma \ref{Lemmater38} to get an expression for $H_{\n_{E}}(\x_E)$ similar to \eqref{eqter:01} and the conclusion direclty follows.
\end{proof}
The previous Proposition \ref{Propter04} directly implies the following corollary.
\begin{corollary}
\label{Corter2}
Let $K$ be a compact subset of $\Omega\subset\R^E$. If the function $F$ satisfies the hypotheses of Proposition \ref{Propter04}, then one can find a constant $C_K$ such that for all $\x_E$ in $K$,
\[|F(\x_E)| \leq C_K\sum_{\n_E\in\mathcal{C}_\B}|\x_E-\y_E|^{\n_E}.\]
\end{corollary}

For instance, let $E = \{1,2,3\}$. Let $F$ be an infinitely differentiable function such that for $(x,y,z)$ in any compact subset $K$ of $\R^3$,
\[|F(x,y,z)| \leq  x^2+y^2,\quad |F(x,y,z)| \leq  y^2+z^2\quand |F(x,y,z)| \leq  x^2+z^2.\]
Then the function $F$ satisfies the hypotheses of Proposition \ref{Propter04} with $\B = \{\{1,2\},\{2,3\},\{1,3\}\}$ and $d=2$. It implies the existence of a constant $C_K$ such that for $(x,y,z)\in K$,
\[|F(x,y,z)|\leq C_K\left(x^2y^2 + y^2z^2 + x^2z^2 + |xyz|\right).\]
\begin{remark}
Let $\Omega^1$ be an open subset of a finite dimensional vector space and assume that $F\in \CC^{0,\infty}(\Omega^1,\Omega)$ (this function space is defined before Lemma \ref{Lemmater24}). Then Proposition \ref{Propter04} remains true if one replace $\CC^\infty(\Omega)$ by $\CC^{0,\infty}(\Omega^1,\Omega)$ and the proof is in all points similar.
\end{remark}
We now apply the previous Corollary \ref{Corter2} to the function $\widetilde{F}_A$.
\begin{lemma}
\label{Lemmater40}
Let $\I$ be a partition of $A$ and $K$ be a compact subset of $\widetilde{\M}_A(\R)\times\Ss_{2A}^+(\R)$. Then there is a constant $C_K$ such that for all $((\widetilde{M},\widetilde{Q}),\Sigma) \in K$ with $\widetilde{M} = \widetilde{M}_\I\quad \widetilde{Q} = \widetilde{Q}_\I$, and
\[\forall B\subset A,\;\;\forall I\in\I,\quad M^B_I = M^{I\cap B}\quand  Q^B_I = Q^{I\cap B},\]
one has
\[\left|\widetilde{F}_A((\widetilde{M},\widetilde{Q}),\Sigma)\right|\leq C_K\sum_{G\in \G_\I} \prod_{\{I,J\}\in E(G)} \|\Sigma_{I,J}\|.\]
\end{lemma}
\begin{proof}
The proposition is trivial if $\I=\{A\}$, and we can assume that $\I\neq\{A\}$. The proof is an application of Corollary \ref{Corter2}. To this end, we define for subsets $B,C$ of $A$ the set 
\[B\circ C = \enstq{\vphantom{\int^a}\{(k,b),(l,c)\}}{k,l\in \{1,2\}\,,\quad b\in B, c\in C}.\]
Then the set $V = \Ss_{2A}(\R)$, endowed with its canonical basis, can be naturally identified with $\R^{A\circ A}$. For $\J\in\Pt_A$ with $\J\succeq\I$ we define
\[B_{\J} = \enstq{\vphantom{\int^a}\{(k,a),(l,b)\}\in A\circ A}{[a]_\J\neq [b]_\J},\]
and
\[\mathcal{B}_\I = \enstq{\vphantom{\int^a}B_{\J}}{\I\preceq \J\prec \{A\}}.\]
The set $B_\J$ encodes the indices of the coefficients in the off-diagonal blocks with respect to the partition $\J$. As a consequence of Lemma \ref{Lemmater24} the function $\widetilde{F}_A$ is in $\CC^{0,\infty}(\widetilde{\M}_A^*(\R),\Ss_{2A}^+(\R))$. Let $\J$ be a partition such that $\I\preceq \J\prec \{A\}$. The assumption on $\widetilde{M}$ and $\widetilde{Q}$ imply that
\[\forall B\subset A,\quad \forall J\in\J,\quad M_J^B = M^{J\cap B}\quand Q_J^B = Q^{J\cap B}.\]
According to Lemma \ref{Lemmater33} there is a constant $C_K$ (that may change from line to line) such that
\begin{align*}
\left|\widetilde{F}_A((\widetilde{M},\widetilde{Q}),\Sigma)\right|&\leq C_K\sup_{\substack{I,J\in \J\\ I\neq J}}\|\Sigma_{I,J}\|^2\\
&\leq C_K\sup_{\{(k,a),(l,b)\}\in B_{\J}}|\Sigma_{ab}^{kl}|^2.
\end{align*}
The function $\widetilde{F}$ then satisfies the hypotheses of Corollary \ref{Corter2} with $d=2$ and family of subsets $\mathcal{B}_\I$, from which we deduce the existence of a constant $C_K$ such that
\begin{align*}
\left|\widetilde{F}_A((\widetilde{M},\widetilde{Q}),\Sigma)\right| &\leq C_K\sum_{\n\in\mathcal{C}_{\B_\I}}\prod_{\substack{e\in A\circ \!A\\ e=\{(k,a),(l,b)\}}}|\Sigma_{ab}^{kl}|^{n_e}\\
&\leq C_K\sum_{\n\in\mathcal{C}_{\B_\I}}\prod_{\substack{e\in A\circ\! A\\ e=\{(k,a),(l,b)\}}}\|\Sigma_{[a]_\I,[b]_\I}\|^{n_e}\\
&\leq C_K\sum_{\n\in\mathcal{C}_{\B_\I}}\prod_{\substack{I,J\in \I\\ I\neq J}}\|\Sigma_{I,J}\|^{|\n_{I\circ J}|}.\numberthis\label{eqter:36}
\end{align*}
To every multi-index $\n\in \mathcal{C}_{\B_\I}$ (which can be seen as a symmetric matrix of size $|A|$ with coefficient in $\{0,1,2\}$), we can associate a graph $G^{\n}$ with set of vertices $\I$, and where the multiplicity of the edge $\{I,J\}$ is given by the number $|\n_{I\circ J}|$. From the definition of the set $\mathcal{C}_{\B_\I}$, any partition into two disjoints subsets of the vertices $\I$ in the graph $G^{\n}$ must be linked with at least two edges. it follows that the graph $G^{\n}$ is two-edge connected and subsequently belongs to the set $\G_\I$. Following inequality \eqref{eqter:36}, one has
\[\left|\widetilde{F}_A((\widetilde{M},\widetilde{Q}),\Sigma)\right| \leq C_K\sum_{G\in\G_\I}\prod_{\lbrace I,J\rbrace\in E(G)}\|\Sigma_{I,J}\|.\]
\end{proof}
\section{Asymptotics of the cumulants of the zeros counting measure}
We are now in position to study the asymptotics of the cumulants of the zeros counting measure associated with a sequence of processes $(f_n)_{n\in\NN}$. We first prove that the non-degeneracy condition \eqref{eqter:20} holds uniformly for $n\in \NN$, which allows us to translate the previous Lemma \ref{Lemmater40} to the cumulant Kac density $F_{A,n}$ associated with the sequence $(f_n)_{n\in\NN}$.
\subsection{Uniform non-degeneracy of the covariance matrix}
Up to now, we assumed that the generic process $f$ that we considered satisfied the non-degeneracy condition \eqref{eqter:20}. For stationary Gaussian processes, this non-degeneracy condition is true under very mild assumptions on the process. For non-stationary processes there seems to be no simple conditions that ensure the validity of \eqref{eqter:20}. Nevertheless in our case of interest, we consider a sequence of Gaussian processes that converges in distribution towards a stationary Gaussian process and we are able to prove some uniform non-degeneracy condition in this setting.\jump

In this subsection, $A$ denotes a finite set and $\I$ a partition of $A$. For $n\in\oNN$, we consider $f_n$ a Gaussian process defined on $nU$. We will use the notations introduced in Section \ref{subsecter3}. In the following, we fix a positive number $\eta$ and we consider the subsequent partition $(\Delta_{\I,\eta})_{\I\in \Pt_A}$ of $(nU)^A$. In particular we consider the quantities $\rho_{A,n}$, $F_{A,n}$, $\Sigma^\I_n(\x_A)$, etc. associated with the process $f_n$. \jump

We assume for now that the sequence $(f_n)_{n\in\oNN}$ satisfies hypotheses $H_1(q)$ and $H_2(q)$ defined in \eqref{eqter:64} and \eqref{eqter:65}, with $q=|A|-1$. In particular the quantity $\Sigma^\I_n(\x_A) = \Cov((f_n)_\I[\x_A])$ is well defined. Since the function $g$ of hypothesis $H_2(q)$ decreases to zero, then for $\varepsilon>0$, one can find a constant $T_\varepsilon$ such that
\[\enstq{x\in\R}{g(x)\geq \varepsilon}\subset [-T_\varepsilon,T_\varepsilon].\numberthis\label{eqter:37}\]
The main proposition of this section is the following.
\begin{proposition}
\label{Propter03}
In the above setting, there is a compact set $K_\eta$ of $\Ss_A^+(\R)$ such that for all $n\in\oNN$ large enough, and $\x_A\in \Delta_{\I,\eta}$, the matrix $\Sigma^\I_n(\x_A)$ lives in $K_\eta$.
\end{proposition}

We prove first Proposition \ref{Propter03} for the limit stationary process $f_\infty$.
\begin{lemma}
\label{Lemmater41}
In the above setting, there is a compact set $K_\eta$ of $\Ss_A^+(\R)$ such that for all $\x_A\in \Delta_{\I,\eta}$, the matrix $\Sigma^\I_\infty(\x_A)$ lives in $K_\eta$.
\end{lemma}
\begin{proof}
According to Lemma \ref{Lemmater02}, one must show the existence of positive constants $C_\eta$ and $c_\eta$ such that
\[\forall \x_A\in \Delta_{\I,\eta},\quad \|\Sigma_\infty^\I(\x_A)\|\leq C_\eta\quand \det \Sigma_\infty^\I(\x_A) \geq c_\eta.\]
From the Hermite Genocchi formula \ref{Lemmater07} and Lemma \ref{Lemmater16}, we observe that the coefficients of the matrix $\Sigma_\infty^\I(\x_A)$ are bounded by $\|g\|_\infty$.
It remains to prove the uniform positiveness of $\det \Sigma_\infty^\I(\x_A)$ on $\Delta_{\I,\eta}$. \jump

The covariance function of $f_\infty$ decreases to zero by assumption. Since for Gaussian vectors, decorrelation implies independence, one see that the process $f_\infty$ is weakly mixing, which in turn implies ergodicity. By Maruyama theorem (see \cite{Mar49}), the spectral measure $\mu_\infty$ of $f_\infty$ has no atoms. It is then a standard fact that this observation implies the non-degeneracy condition \eqref{eqter:20}, and Lemma \ref{Lemmater17} implies that the Gaussian vector $(f_\infty)_\I[\x_A]$ is also non-degenerate for $\x_A\in \Delta_{\I,\eta}$.\jump
We prove the uniform lower bound for $\x_A\in\Delta_{\I,\eta}$ by induction on the size of the set $A$. If $|A| = 1$ it reduces to the fact that the process $f_\infty$ is non-degenerate. Assume that the property is true for every strict subset $B$ of $A$. Let $\J$ be another partition of $A$ such that $\J\succeq \I$, and $\varepsilon>0$. Following Equation \eqref{eqter:22} we introduce
\[\Delta_{\J,T_\varepsilon} = \enstq{\x_A\in U^A}{\J_{T_\varepsilon}(\x_A) = \J}.\]
We can assume that $T_\varepsilon \geq |A|\eta$. In that case, one has
\[\Delta_{\I,\eta} \subset \bigsqcup_{\J\succeq \I} \Delta_{\J,T_\varepsilon}.\]
In the case $\J = \{A\}$, one has for all $a,b\in A$ and $\x_A\in \Delta_{\{A\},T_\varepsilon}$ 
\[|x_a-x_b|\leq |A|\,T_\varepsilon.\]
The set $\Delta_{\I,\eta} \cap \Delta_{\{A\},T_\varepsilon}$ is not compact, but it is compact by translation in the sense that it is compact if one fixes one of the coordinates. This compactness property plus the stationarity of the process $f_\infty$ implies that one can find a positive constant $c_{\eta,\varepsilon}$ such that
\[\forall \x_J\in \Delta_{\I,\eta} \cap \Delta_{\{A\},T_\varepsilon},\quad\det \Sigma^\I(\x_A) \geq c_{\eta,\varepsilon}.\]
Now assume that $\J\neq \{A\}$. If $\x_A\in \Delta_{\J,T_\varepsilon}$ then for $a,b\in A$ such that $[a]_\J\neq [b]_\J$, and $u,v\leq |A|-1$,
\[|r_\infty^{(u,v)}(x_a-x_b)| \leq \varepsilon.\]
It implies that,
\[\sup_{\substack{I,J\in \J\\I\neq J}}\|\Sigma_\infty^\I(\x_A)_{I,J}\|\leq \varepsilon,\]
and thus
\[\|\Sigma_\infty^\I(\x_A) - (\Sigma_\infty^\I(\x_A))_\J\|\leq \varepsilon.\]
Since the determinant is a smooth function of the matrix coefficients and the matrix $\Sigma_\infty^\I(\x_A)$ is bounded, we deduce the existence of a constant $C_\eta$ such that for $\x_A\in \Delta_{\J,T_\varepsilon}$,
\begin{align*}
\left|\det \Sigma_\infty^\I(\x_A) - \det \,(\Sigma_\infty^\I(\x_A))_\J\right|&= \left|\det \Sigma_\infty^\I(\x_A) - \prod_{J\in\J}\det \Sigma_\infty^{\I_J}(\x_J)\right|\\
&\leq C_\eta\varepsilon,
\end{align*}
and thus
\[\det \Sigma_\infty^\I(\x_A)\geq \prod_{J\in\J}\det \Sigma_\infty^{\I_J}(\x_J) - C_\eta\varepsilon.\]
For all $J\in\J$, the set $J$ is a strict subset of $A$. Moreover, if $\x_A\in\Delta_{\I,\eta}$ then $\x_J\in\Delta_{\I_J,\eta}$. By induction hypothesis, one can find a positive constant $c_\eta$ depending only on $\eta$ such that $\det \Sigma_\infty^{\I_J}(\x_J) \geq c_\eta$ when $\x_J\in\Delta_{\I_J,\eta}$. It implies that 
\[\forall \x_A\in \Delta_{\I,\eta} \cap \Delta_{\J,\varepsilon},\quad \det \Sigma^\I_\infty(\x_A) \geq(c_\eta)^{|\J|} - C_\eta\varepsilon.\]
Taking $\varepsilon$ small enough and gathering the case $\J = \{A\}$ and $\J\neq \{A\}$, the conclusion follows.
\end{proof}
\paragraph{Proof of Proposition \ref{Propter03}.}\strut\jump

\begin{proof}
A reformulation of hypothesis $H_1(q)$ applied to the compact set $[-T_\varepsilon,T_\varepsilon]$ yields
\[\lim_{n\rightarrow +\infty} \sup_{\substack{s,t\in nU\\|s-t|\leq T_\varepsilon}} \left|r_n^{(u,v)}(s,t) - \psi\left(\frac{s}{n}\right) r_\infty^{(u,v)}(s,t)\right| = 0.\]
The function $\psi$ is uniformly continuous by hypothesis and we can define $\omega_\psi$ its uniform modulus of continuity. By hypothesis, there are positive constants $c_\psi$ and $C_\psi$ such that for all $x\in U$,
\[c_\psi\leq \psi(x)\leq C_\psi.\numberthis\label{eqter:38}\]
Let $n\in\oNN$. If $s,t\in nU$ and $|t-s|>T_\varepsilon$ then hypothesis $H_2(q)$ implies that for $u,v\leq |A|-1$,
\[|r_n^{(u,v)}(s,t)| \leq \varepsilon.\numberthis\label{eqter:39}\]
Gathering \eqref{eqter:38} and \eqref{eqter:39}, there is $n_\varepsilon\in \NN$ such that for $n\geq n_\varepsilon$, and $s,t\in nU$
\[\left|r_n^{(u,v)}(s,t) - \psi\left(\frac{s}{n}\right) r_\infty^{(u,v)}(s,t)\right|\leq \varepsilon(1+C_\psi).\numberthis\label{eqter:40}\]
Let $n\in\oNN$ with $n\geq n_\varepsilon$ and $\x_A\in \Delta_{I,\eta}$. For $I,J\in\I$, $a\in I$ and $b\in J$ one has from Lemma \ref{Lemmater16}
\[\Sigma_n^\I(\x_A)_{ab} = \int_{C_{a|I}\times C_{b|J}}r_n^{(|\,a|I\,|-1\,,\,|\,b|J\,|-1)}\left(\sum_{i\in a|I}s_ix_i,\sum_{j\in b|J}t_jx_j\right)\dd m(\underline{s}_{a|I})\dd m(\underline{t}_{b|J}).\]
Inequality \eqref{eqter:40} implies
\[\left|\Sigma_n^\I(\x_A)_{ab}- \psi\left(\frac{x_a}{n}\right)\Sigma_\infty^\I(\x_A)_{ab}\right| \leq \varepsilon(1+C_\psi) + R_n(a,b),\numberthis\label{eqter:42}\]
where
\begin{align*}
R_n(a,b) &= \|g\|_\infty\int_{C_{a|I}\times C_{b|J}}\left|\psi\left(\frac{x_a}{n}\right)-\psi\left(\frac{\sum_{i\in a|I}s_ix_i}{n}\right)\right|\dd m(\underline{s}_{a|I})\dd m(\underline{t}_{b|J})\\
&\leq C\sup_{\underline{s}_{a|I}\in C_{a|I}}\left|\psi\left(\frac{x_a}{n}\right)-\psi\left(\frac{\sum_{i\in a|I}s_ix_i}{n}\right)\right|.
\end{align*}
For $i\in I$, one has $|x_a-x_i|\leq |A|\eta$. It implies that for any convex combination $y$ of the variables $(x_i)_{i\in I}$ one also have $|x_a-y|\leq |A|\eta$. We deduce that
\[R_n(a,b)\leq C\omega_\psi\left(\frac{|A|\eta}{n}\right).\]
There is $n_{\eta,\varepsilon}$ such that for $n\geq n_{\eta,\varepsilon}$,
\[R_n(a,b)\leq \varepsilon,\]
and thus coming back to inequality \eqref{eqter:42},
\[\left|\Sigma_n^\I(\x_A)_{ab}- \psi\left(\frac{x_a}{n}\right)\Sigma_\infty^\I(\x_A)_{ab}\right| \leq \varepsilon(2+C_\psi).\numberthis\label{eqter:51}\]
It implies the existence of a constant $C_\eta$ such that for $n$ large enough and $\x_A\in \Delta_{\I,\eta}$,
\[\left|\det \Sigma_n^\I(\x_A) - \left(\prod_{a\in A}\psi\left(\frac{x_a}{n}\right)\right)\det \Sigma^\I_\infty(\x_A)\right|\leq  C_\eta\varepsilon.\]
We deduce that
\[\det \Sigma_n^\I(\x_A)\geq c_\psi^{|A|}\det \Sigma_\infty^\I(\x_A)-C_\eta\varepsilon.\]
The conclusion follows from the previous Lemma \ref{Lemmater41} covering the stationary case, and taking $\varepsilon$ small enough.
\end{proof}
As a consequence of the previous Proposition \ref{Propter03}, we deduce the following corollary about convergence of the Kac density associated with the process $f_n$.
\begin{corollary}
\label{Corter3}
Assume that the sequence of processes $(f_n)_{n\in\oNN}$ satisfies hypotheses $H_1(q)$ and $H_2(q)$ defined in \eqref{eqter:64} and \eqref{eqter:65}, with $q=2|A|-1$. Then there is a compact set $K_\eta$ of $\Ss_{2A}^+(\R)$ such that for all $n\in\oNN$ large enough and $\x_A\in \Delta_{\I,\eta}$,
\[\Sigma^\I_n(\x_{A,A})\in K_\eta.\]
In particular we have the following convergence, uniformly for $x\in U$ and $\underline{t}_A$ in compact subsets of $\R^A$
\[\lim_{n\rightarrow+\infty} \rho_n(nx+\underline{t}_A) = \rho_\infty(\underline{t}_A)\quand \lim_{n\rightarrow+\infty} F_{A,n}(nx+\underline{t}_A) = F_{A,\infty}(\underline{t}_A).\]
\end{corollary}
\begin{proof}
The first assertion is a direct application of Proposition \ref{Propter03} with the set $2A$, using the fact that $\x_{A,A}\in\Delta_{\I,\eta}$ when $\x_A\in\Delta_{\I,\eta}$. As for the second one, the proof of Proposition \ref{Propter03}, and in particular equation \eqref{eqter:51}, implies that for all partition $\I$ of $A$, one has the following convergence, uniformly for $x\in U$ and $\underline{t}_A$ in a bounded subset of $\Delta_{\I,\eta}$
\[\lim_{n\rightarrow +\infty} \Sigma^\I_n(nx+\underline{t}_{A,A}) = \psi(x)\Sigma^\I_\infty(\underline{t}_{A,A}).\]
The conclusion follows from the alternative expression for $\rho_n$ given by Lemma \ref{Lemmater22}. Note that the function $\rho_\infty$ does not depends on the function $\psi(x)$, by a change of variable.
\end{proof}
\subsection{Asymptotics of the cumulants}
Let $A$ be a finite set of cardinal $p$. We assume that the sequence of processes $(f_n)_{n\in\oNN}$ satisfies hypotheses $H_1(q)$ and $H_2(q)$ defined in \eqref{eqter:64} and \eqref{eqter:65}, with $q=2p-1$. We then choose $\eta = \frac{\omega}{2p}$ where $\omega$ is the parameter of hypothesis $H_2(q)$, so that
\[g_\omega(x) = \sup_{|u|\leq 2\eta p}g(x+u).\]
\subsubsection{Decay of the cumulant Kac density}
Let us now translate Lemma \ref{Lemmater40} to the cumulant Kac density $F_{A,n}$. The previous Corollary \ref{Corter3} ensures that the matrix $\Sigma_n^\I(\x_{A,A})$ lives in a compact subset of $\Ss_{2A}^+(\R)$ when $\x_A\in\Delta_{\I,\eta}$ and $n$ is large enough.
\begin{lemma}
\label{Lemmater42}
There is a constant $C$ such that for all $\x_A \in (nU)^A$,
\[\left|F_{A,n}(\x_A)\right|\leq C\sum_{G\in \G_A} \prod_{\{a,b\}\in E(G)} g_\omega(x_a-x_b).\]
\end{lemma}
\begin{proof}
Let $\I$ be a partition of $A$ and $\x_A\in \Delta_{\I,\eta}$. According to Corollary \ref{Corter3}, the matrix $\Sigma_n^\I(\x_{A,A})$, for $n$ large enough depending only on $\eta$, lives in a compact subset of $\Ss_{2A}^+(\R)$ depending only on the parameter $\eta$. By Proposition \ref{Propter01}, the element $(\widetilde{M}^\I(\x_A),\widetilde{Q}^\I(\x_A))$ also lives in a compact subset of $\widetilde{\M}_A^*(\R)$ that depends only on $\eta$. We can then apply Lemma \ref{Lemmater40} with $\Sigma = \Sigma_n^\I(\x_{A,A})$ and $(\widetilde{M},\widetilde{Q}) = (\widetilde{M}^\I(\x_A),\widetilde{Q}^\I(\x_A))$. Given the representation formula for $F_A$ given by Lemma \ref{Lemmater30}, we deduce the existence of a constant $C_\eta$ such that for all $\x_A\in \Delta_{\I,\eta}$,
\[\left|F_A(\x_A)\right|\leq C_\eta\sum_{G\in \G_\I} \prod_{\{I,J\}\in E(G)} \|(\Sigma_n^\I(\x_{A,A}))_{I,J}\|.\]
Let $H\in\G_\I$. According to Lemma \ref{Lemmater20}, there is a graph $G\in\G_A$ such that $G_\I=H$.
If we remove the edges $\{a,b\}$ of $G$ such that $[a]_\I = [b]_\I$, then there is a bijection between the edges of $G$ and the edges of $H$ given by the mapping
\[\{a,b\}\longrightarrow \{[a]_\I,[b]_\I\}.\]
Let $\{a,b\}$ an be edge of the graph $G$.
\begin{itemize}
\item If $[a]_\I = [b]_\I$ then $|x_a-x_b|\leq |A|\eta$. We deduce that
\[0<r_\infty(0)\leq g(0)\leq g_\omega(x_a-x_b).\]
\item If $[a]_\I \neq [b]_\I$ then from the Hermite-Genocchi formula and Lemma \ref{Lemmater16},
\[\|(\Sigma_n^\I(\x_{A,A}))_{I,J}\| \leq \sup_{|s|\leq 2\eta p} g(x_a-x_b + s)\leq g_\omega(x_a-x_b).\]
\end{itemize}
We deduce that
\begin{align*}
\prod_{\{I,J\}\in E(H)} \|(\Sigma_n^\I(\x_{A,A}))_{I,J}\| &\leq \prod_{\substack{\{a,b\}\in E(G)\\ [a]_\I\neq [b]_\I}} g_\omega(x_a-x_b)\\
&\leq C\prod_{\{a,b\}\in E(G)} g_\omega(x_a-x_b).
\end{align*}
We deduce the existence of a constant $C_\eta$ such that for $\x_A\in\Delta_{\I,\eta}$,
\[\left|F_A(\x_A)\right|\leq C_\eta \sum_{G\in \G_A} \prod_{\{a,b\}\in E(G)} g_\omega(x_a-x_b).\]
The inequality is true for every partition $\I$ of $A$ and the conclusion follows.
\end{proof}
\subsubsection{Convergence of the error term towards zero}
Recall from Definition \eqref{eqter:52} that $\nu_n$ is the random counting measure of the zero set of the Gaussian process $f_n(n\,.\,)$ defined on $U$. The previous Lemma \ref{Lemmater42} and the formula for the $p$-th cumulant given by Proposition \ref{Propter02} shows that the convergence of the cumulant reduces to the behavior of the quantity
\[I_n(G) = \int_{(nU)^A}\prod_{a\in A}\left|\phi_a\left(\frac{x_a}{n}\right)\right|\prod_{e=\{i,j\}\in E(G)} g_e(x_i-x_j)\dd\x_A,\numberthis\label{eqter:61}\]
where $G$ is a $2$-edge connected graph with set of vertices $A$ and set of edges $E(G)$, $(\phi_a)_{a\in A}$ are bounded functions with compact support and $(g_e)_{e\in E(G)}$ are even functions in $L^2\cap L^\infty$.\jump

The quantity $I_n(G)$, in the context of cumulants asymptotics, is somehow reminiscent of a theorem of Szegő (see \cite{Avr10} and the references therein), where this kind of integral received a thorough treatment and Hölder bounds that depends on the structure of the graph were given. It has been for instance used conjointly with diagram formulas and Wiener Chaos expansion techniques to prove the Gaussian asymptotics of non-linear functional of random measures, see for instance \cite{Pec11}.\jump

Nevertheless our setting is not exactly the same, and we were able to give a very short and self contained argument, that relies only on a basic interpolation inequality for Hölder norms, which proves a tight Hölder type bound for the quantity $I_n(G)$ when $G$ is $2$-edge connected. We recall the following fundamental theorem about Hölder inequality proved by Barthe in \cite[Sec. 2]{Bar2}, which is a particular instance of the Brascamp--Lieb inequality (a good survey reference is for instance \cite{Ben08}).
\begin{theorem}[Hölder interpolation]
\label{Theoremter4}
Let $m,n$ positive integers, and $v_1,\ldots,v_m$ be non-zeros vectors which span the Euclidean space $\R^n$. We denote by $Q$ the subset of $[0,1]^m$ such that $\underline{q}\in Q$ if there is a finite constant $C_{\underline{q}}$ such that for every measurable functions $\psi_1,\ldots,\psi_m$ from $\R$ to $\R$,
\[\int_{\R^n} \prod_{i=1}^m |\psi_i(\langle \x,v_i\rangle)|\dd \x \leq C_{\underline{q}} \prod_{i=1}^m \left(\int_\R |\psi_i(x)|^{1/q_i}\dd x\right)^{q_i}.\]
Then $Q$ is convex.
\end{theorem}
The above theorem implies the following theorem about the integral quantity $I_n(G)$. Recall that $p=|A|$.
\begin{lemma}
\label{Lemmater51}
Assume that for all $e\in E(G)$, $g_e \in L^\frac{p}{p-1}$. Then for every $e\in E(G)$, there is a number $p_e\geq p/(p-1)$ such that
\[\frac{1}{n}I_n(G)\leq \left(\prod_{a\in A}\|\phi_a\|_p\right)\left(\prod_{e\in E(G)}\|g_e\|_{p_e}\right).\]
Assume that $p\geq 3$ and $g_e \in L^2\cap L^\infty$. Then
\[\lim_{n\rightarrow +\infty} \frac{1}{n^{p/2}}I_n(G) = 0.\]
\end{lemma}
\begin{proof}
Let $(T_a)_{a\in A}$ be the family of spanning trees of $G$ given by Lemma \ref{Lemmater50}. For fixed index $a\in A$, the linear mapping
\[\x_A\longmapsto (x_a, (x_b-x_c)_{\{b,c\}\in E(T_a)})\]
is volume preserving. For $e\notin E(T_a)$ we bound the term $g_e(x_i-x_j)$ in $I_n(G)$ by $\|g_e\|_\infty$, and for $b\neq a$, the function $\phi_b$ by $\|\phi_b\|_\infty$. By a change of variable, we get
\[I_n(G) \leq n\|\phi_a\|_1\left(\prod_{b\neq a}\|\phi_b\|_\infty\right)\left(\prod_{e\in E(T_a)}\|g_e\|_1\right)\left(\prod_{e\notin E(T_a)}\|g_e\|_\infty\right).\]
This inequality is true for all $a\in A$. By Theorem \ref{Theoremter4}, one can interpolate this collection of Hölder inequalities indexed by the set $A$ and convex combination $(1/p,\ldots,1/p)$ to obtain
\[I_n(G) \leq Cn\left(\prod_{a\in A}\|\phi_a\|_p\right)\left(\prod_{e\in E(G)}\|g_e\|_{p_e}\right),\quad\text{with}\quad \frac{1}{p_e} = \frac{1}{p}\sum_{a\in A}\one_{E(T_a)}(e).\numberthis\label{eqter:59}\]
Since for all $e\in E(G)$, there is a tree $T_{a_e}$ that does not contain the edge $e$, one must have $p_e \geq \frac{p}{p-1}$, and the first part of the lemma follows. For the second part, note that we also have the crude bound
\[I_n(G)\leq n^p\left(\prod_{a\in A}\|\phi_a\|_1\right)\left(\prod_{e\in E(G)}\|g_e\|_{\infty}\right).\]
We can once again interpolate this inequality with inequality \eqref{eqter:59} and convex combination $\left(\frac{p}{2(p-1)}, \frac{p-2}{2(p-1)}\right)$ to get
\[\frac{1}{n^{p/2}}I_n(G)\leq C\left(\prod_{a\in A}\|\phi_a\|_2\right)\left(\prod_{e\in E(G)}\|g_e\|_{q_e}\right),\quad\text{with}\quad q_e = 2p_e\frac{p-1}{p}\geq 2.\numberthis\label{eqter:60}\]
It remains show that the left hand side of \eqref{eqter:60} converges towards zero for $p\geq 3$. If the functions $(g_e)_{e\in E(G)}$ are bounded and compactly supported, then inequality \eqref{eqter:59} implies the convergence towards zero of the left hand side of \eqref{eqter:60} when $p\geq 3$. In the general case, one can take, for every $e\in E(G)$, a sequence of bounded and compactly supported functions that converge towards $g_e$ in $L^{q_e}$. The Hölder bound given by \eqref{eqter:60} and the triangular inequality imply the desired result.
\end{proof}
The above Lemma \ref{Lemmater51} implies that the space of test functions $(\phi_a)_{a\in A}$ can be extended to $L^p(U)$. The previous Lemma \ref{Lemmater51} and the convergence of the Kac density given by Corollary \ref{Corter3} imply the following lemma.
\begin{lemma}
\label{Lemmater53}
For all $p\geq 3$,
\[\lim_{n\rightarrow+\infty}\;\frac{1}{n^{p/2}}\int_{(nU)^A}\underline{\phi}_A^\otimes\left(\frac{\x_A}{n}\right)F_{A,n}(\x_A)\dd\x_A=0.\]
For all $p\geq 1$, if $g_\omega \in L^\frac{p}{p-1}$ then
\[\lim_{n\rightarrow+\infty} \frac{1}{n}\int_{(nU)^A}\underline{\phi}_A^\otimes\left(\frac{\x_A}{n}\right)F_{A,n}(\x_A)\dd\x_A = \left(\int_{U} \prod_{a\in A}\phi_a(y)\dd y\right)\left(\int_{\R^{p-1}} F_{A,\infty}(0,\x)\dd \x\right).\]
\end{lemma}
\begin{proof}
According to Lemma \ref{Lemmater42}, there is a constant $C$ such that
\[\left|\int_{(nU)^A}\underline{\phi}_A^\otimes\left(\frac{\x_A}{n}\right)F_{A,n}(\x_A)\dd\x_A\right|\leq C\sum_{G\in \G_A} I_n(G),\]
where $I_n(G)$ is defined in $\eqref{eqter:61}$ with functions $g_e = g_\omega$. The first part of the corollary is an immediate consequence of the second part of Lemma \ref{Lemmater51}. Assume first that the functions $(\phi_a)_{a\in A}$ are continuous and compactly supported. In that case, pick $a_0\in A$. We define $y_{a_0} = 0$ and we make the change of variables
\[x_{a_0} = ny\quand \forall a\in A\setminus\{a_0\},\;\;x_a = ny + y_a.\]
Then we have the following uniform convergence for $y\in U$ and $\y_A$ in compact subsets of $\R^A$
\[\lim_{n\rightarrow+\infty} \phi_a\left(y + \frac{y_a}{n}\right) = \phi_a(y),\]
and according to Corollary \ref{Corter3},
\[\lim_{n\rightarrow+\infty} F_{A,n}\left(ny+\y_A\right) = F_{A,\infty}(\y_A).\]
The conclusion then follows from the dominated convergence theorem. In the general case, we consider for all $a\in A$ a sequence of continuous and compactly supported functions that converges towards $\phi_a$ in $L^p$. The Hölder bound given by Lemma \ref{Lemmater51} and another application of dominated convergence theorem imply the desired result.
\end{proof}
Given the expression of cumulants given by Proposition \ref{Propter02} and the previous Lemma \ref{Lemmater53}, we then deduce the following theorem concerning the convergence of cumulants associated with the linear statistics of the zeros counting measure of the sequence of processes $(f_n)_{n\in \oNN}$.  We define the Stirling number of the second kind
\[\left\lbrace\begin{matrix}
p \\ 
k
\end{matrix} \right\rbrace := \Card\enstq{\I\in \Pt_A}{|\I|=k}.\]
\begin{theorem}
\label{Theoremter5}
Let $p\geq 2$ and assume that the sequence of processes $(f_n)_{n\in \oNN}$ satisfies the hypotheses $H_1(q)$ and $H_2(q)$ with $q=2p-1$. Let $\phi\in L^1\cap L^{p^2}$. If $p\geq 3$ then
\[\lim_{n\rightarrow+\infty} \frac{1}{n^{p/2}}\kappa_p(\langle\nu_n,\phi\rangle) = 0.\]
Moreover when $g_\omega \in L^{\frac{p}{p-1}}$,
\[\lim_{n\rightarrow+\infty} \frac{1}{n}\kappa_p(\langle\nu_n,\phi\rangle) = \left(\int_U\phi^p(y)\dd y\right)\sum_{k=1}^p\left\lbrace\begin{matrix}
p \\ 
k
\end{matrix} \right\rbrace
\left(\int_{\R^{k-1}}\!\!\!F_{k,\infty}(0,\x)\dd \x\right).\]
\end{theorem}
\begin{proof}
Let $p\geq 3$. Recall from Proposition \ref{Propter02} that
\[\kappa_p(\langle\nu_n,\phi\rangle) = \sum_{\I\in\Pt_A} \int_{(nU)^\I}\left(\prod_{I\in\I}\phi\left(\frac{\x_I}{n}\right)^{|I|}\right) F_{\I,n}(\x_\I)\dd\x_\I.\]
Since $\phi\in L^1\cap L^{p^2}$, then for every partition $\I$ of $\{1,\ldots,p\}$ and $I$ a subset of $\I$ one has that the function $\phi^{|I|}$ is in $L^{|\I|}$. According to the previous Lemma \ref{Lemmater53}, one has
\[\frac{1}{n^{p/2}}\left|\kappa_p(\langle\nu,\phi\rangle)\right| \leq \sum_{\I\in\Pt_A} \frac{1}{n^{p/2}}\left|\int_{(nU)^\I}\left(\prod_{I\in\I}\phi\left(\frac{\x_I}{n}\right)^{|I|}\right) F_{\I,n}(\x_\I)\dd\x_\I\right|\underset{n\rightarrow+\infty}{\longrightarrow} 0,\]
which proves the first assertion. As for the second assertion, it is again a consequence of Lemma \ref{Lemmater53}, which implies that
\[\lim_{n\rightarrow+\infty} \frac{1}{n}\kappa_p(\langle\nu_n,\phi\rangle) = \left(\int_U\phi^p(y)\dd y\right)\sum_{\I\in\Pt_A} \left(\int_{\R^{|\I|-1}}\!\!\!F_{|\I|,\infty}(0,\x)\dd \x\right).\]
\end{proof}
The proof of the main Theorem \ref{Theoremter9} is a reformulation of the previous Theorem \ref{Theoremter5}, with
\[\forall p\geq 1,\quad\gamma_p = \sum_{k=1}^p\left\lbrace\begin{matrix}
p \\ 
k
\end{matrix} \right\rbrace
\left(\int_{\R^{k-1}}\!\!\!F_{k,\infty}(\x)\dd \x\right).\]
In particular, one has
\[\gamma_1 = \frac{1}{\pi}\sqrt{\frac{-r_\infty''(0)}{r_\infty(0)}}\quand \gamma_2 = \gamma_1 + \int_\R F_{2,\infty}(0,x)\dd x.\]
It has been shown for instance in \cite{Slu91} that under our assumptions on the process $f_\infty$, the constant $\gamma_2$ is positive, from which follows the central limit theorem for the linear statistic associated with the zeros counting measure.
\newpage
\printbibliography
\end{document}